\documentclass[vecarrow]{svjour3}       
\usepackage{graphicx}
\usepackage[margin=1in]{geometry}  
\usepackage{amsfonts,amsmath,amssymb}
\usepackage{graphics}
\usepackage{epsfig}
\usepackage{graphicx}
\usepackage{multicol}
\usepackage{multirow}
\usepackage{slashbox}
\usepackage{authblk}
\usepackage{kpfonts}

\smartqed
\numberwithin{equation}{section}

\newcommand{\re}[1]{\mathbb{#1}}      

\spnewtheorem{thm}{Theorem}{\bf}{\it}
\spnewtheorem{lem}[thm]{Lemma}{\bf}{\it}
\spnewtheorem{prop}[thm]{Proposition}{\bf}{\it}
\spnewtheorem{assu}[thm]{Assumption}{\bf}{\it}
\spnewtheorem{rem}[thm]{Remark}{\bf}{\it}
\spnewtheorem{def1}[thm]{Definition}{\bf}{\it}
\spnewtheorem{algorithm}[thm]{Algorithm}{\bf}{\it}

\DeclareMathOperator*{\argmin}{arg\,min}

\begin{document}

\title{A Preconditioned Descent Algorithm for Variational Inequalities of the Second Kind Involving the $p$-Laplacian Operator\thanks{Supported in part by the Ecuadorian Secretary of Higher Education, Science, Technology and Innovation, SENESCYT, under the project PIC-13-EPN-001 ``Numerical Simulation of Cardiac and Circulatory Systems'', the Escuela Polit\'ecnica Nacional, under the project PIMI 14-12 ``Numerical Simulation of Viscoplastic Fluids in Food Industry" and the MATH-AmSud Project ``SOCDE-Sparse Optimal Control of Differential Equations: Algorithms and Applications''.}}

\titlerunning{Preconditioned Descent Algorithm for VIs Involving the $p$-Laplacian} 

\author{Sergio Gonz\'alez-Andrade}

\institute{S. Gonz\'alez-Andrade \at
              Research Center on Mathematical Modeling (ModeMat) and \\Department of Mathematics -  Escuela Polit\'ecnica Nacional, Quito\\\small Ladr\'on de Guevara E11-253\\Quito 170525, Ecuador\\
\email{sergio.gonzalez@epn.edu.ec}}

\date{Received: date / Accepted: date}

\maketitle

\begin{abstract}
This paper is concerned with the numerical solution of a class of variational inequalities of the second kind, involving the $p$-Laplacian operator. This kind of problems arise, for instance, in the mathematical modelling of non-Newtonian fluids. We study these problems by using a regularization approach, based on a Huber smoothing process. Well posedness of the regularized problems is proved, and convergence of the regularized solutions to the solution of the original problem is verified.  We propose a preconditioned descent method for the numerical solution of these problems and analyze the convergence of this method in function spaces. The existence of admissible descent directions is established by variational methods and admissible steps are obtained by a backtracking algorithm which approximates the objective functional by polynomial models. Finally, several numerical experiments are carried out to show the efficiency of the methodology here introduced. 
\end{abstract}

\section{Introduction}\label{intro}
Variational inequalities (VIs) provide a versatile background for the analysis and modelling of physical phenomena which involve free boundary problems. This kind of problems include, for instance, contact of rigid bodies, flow of electro- and magneto-rheological fluids and flow of viscoplastic materials, among others (see \cite{jidiaz,dlRGpf,dlRG2D,dlRHint,Huilgol}). The wide range of applications of this kind of free boundary problems make the analysis and the numerical simulation of their associated VIs a quite interesting and challenging field of research.

On the other hand, the $p$-Laplacian operator has been widely analysed as a model case for quasilinear and degenerate elliptic equations (see \cite{jidiaz,Struwe}). Regarding the $p$-Laplacian problem, several analytical results, concerning existence and multiplicity of solutions, have been obtained in, \textit{e.g.}, \cite{Simon1,Struwe}. Further, the numerical analysis of the $p$-Laplacian problem has been a productive field of research. Mainly, the finite element approximation has been broadly studied in the literature in, for example, \cite{barret,glomaro,Huang} and the references therein. The numerical realisation of the $p$-Laplacian has been carried out by the Augmented Lagrangian method \cite{glomaro} and, recently, by using optimization and variational techniques \cite{Huang} and multigrid algorithms \cite{bermejomg}. 

In spite of the fact that the $p$-Laplacian operator is a extensively studied field, scarce work can be found in the numerical analysis of variational inequalities involving this differential operator. The obstacle problem with associated operator of the $p$-Laplacian type has been analyzed in \cite{liubarret}. There, the authors analize a finite element approximation of the problem and provide error estimates for such  approximation. In \cite{joubue}, the obstacle problem in the context of a glaceology application has been studied. The authors consider a variational inequality of the first kind, involving the $p$-Laplacian operator. They analyze the existence, uniqueness and regularity of solutions, and propose a finite element approximation of the problem. Elliptic and parabolic quasi-variational inequalities involving quasilinear operators are considered in \cite{hinrau1,hinrau2}. In these papers, the authors propose and study a semismooth Newton approach for the numerical solution of these problems, and provide several theoretical results regarding existence and regularity of solutions. Finally, in \cite{Huang}, the authors propose a finite dimensional descent algorithm for the numerical solution of  several differentiable problems, including the classical Dirichlet $p$-Laplacian problem and a class of variational inequalities wich combines the Laplacian and the $p$-Laplacian operators, for $1<p<\infty$. 

In contrast with the previous contributions, in this paper we are concerned with variational inequalities of the second kind. The importance of this class of variational inequalities lies on the fact that they can be used to model the flow of a particular class of viscoplastic materials: the Herschel-Bulkley fluids. 

Herschel-Bulkley is a power-law model with plasticity. This model is used to simulate some materials whose behaviour depends on the flow index $p$. This constant measures the degree to which the fluid is shear-thinning ($1<p<2$) or shear-thickening ($p>2$). The Herschel-Bulkley model can be seen as a generalization of the classical Bingham model, which is retrieved from the first one by taking $p=2$ (\cite{dlRGpf,dlRG2D}). Depending on the value of the power index, this model can be used to simulate a wide range of materials from nail polish or whipped cream (shear-thinning fluids) to quicksand or silly putty (shear-thickening fluids) (see \cite{Chhabra}).  Furthermore, the Herschel-Bulkley model has proved to be accurate in the modelling of blood, a known shear-thinning fluid \cite{Quart-blood,Sankar-blood,Shah-blood}.


In consequence, the numerical resolution of VIs involving a $p$-Laplacian operator is an important research field. A classical approach to these problems is the Augmented Lagrangian method (see \cite{Huilgol}), while, from our point of view, the application of optimization and variational techniques has not been explored enough in this context. Several optimization problems involving non differentiable functionals have been successfully analyzed by using this approach (see \cite{DelosR}). Further, the analysis of this kind of methods in function spaces is a challenging but promising research field.

As stated before, in this paper we are concerned with a class of variational inequalities of the second kind involving the $p$-Laplacian operator and the $L^1$-norm of the gradient. Our main intention is to develop an efficient algorithm for the numerical solution of these problems. The main challenge in this aim consists in designing a numerical strategy, which allows to obtain an accurate solution with a fast convergent method.  In the context of numerical solution of VIs of the second kind, the local smoothing techniques, such as Huber regularization, have proved to be an effective way to achieve such a goal (see \cite{dlRGpf,dlRG2D}).  Therefore, we study the variational inequality as an equivalent minimization problem of a non differentiable functional, and we regularize the functional by a Huber procedure. Further, the convergence of the regularized solutions to the original one is established. 

For the numerical solution of the VIs under study, we propose a preconditioned descent algorithm in function spaces. Several issues arise in this approach. Mainly, we need to discuss the existence of admissible search directions and admissible step sizes. Admissibility of step sizes depends on the line search strategy. Here, we propose a backtracking algorithm which approximates the objective functional by polynomial models. In this way, the algorithm provides admissible step sizes with low computational effort. 

The existence of admissible search directions is analyzed considering the two cases $1<p<2$ and $p>2$, separately. In the case $1<p<2$, we first discuss the properties of a suitable Hilbert space in which we will propose and analyse the algorithm. Next, we define the preconditioner and prove the existence of admissible search directions. This is achieved by discussing the conditions for the Zoutendijk condition to hold (\cite{nocacta,SunYuan}). Finally, we state and prove a global convergence result for this algorithm.  The case $p>2$ poses analytical issues which prevent us from studying the algorithms in function spaces. In fact, it is not possible to prove existence of admissible search directions in the same function space in which the elliptic preconditioner is defined. We discuss in detail these issues and propose an alternative algorithm in a finite element space. Next,  we state and prove a global convergence result for this algorithm as well. 

Though the descent algorithms are usually slow, in our case the design of suitable preconditioners and the use of an innovative line search algorithm help us to obtain a robust algorithm which only needs the solution of one linear system per iteration. Further, since the algorithms are proposed, at least in the $1<p<2$ case, in function spaces, they are expected to exhibit mesh independence.

The paper is organized as follows. In Section \ref{sec:statandreg} we introduce and analyze the variational inequality and its associated optimization problem. Next, by using Fenchel’s duality theory, a necessary condition is derived and,  since the original problem is ill-posed, a family of regularized optimization problems is introduced and the convergence of the regularized solutions to the original one is proved. In Section \ref{sec:preconalgo} the numerical approach to these problems is studied. We propose preconditioned descent algorithms for the regularized optimization problems, considering separately the two cases $1<p<2$ and $p> 2$. Particularly, we  prove a global convergence result for all the algorithms constructed. Section \ref{sec:numres} is devoted to the numerical experience. First we discuss the main issue regarding the implementation of our algorithms. Mainly, the discretization issues and the implementation of the line-search methods. Next, several numerical experiments, which illustrate the main features of the proposed approach, are carried out. Finally, in Section \ref{sec:conclusions}, we outline conclusions on this work and discuss some challenging issues that can be analyzed in future contributions.

\section{Problem Statement and Regularization}\label{sec:statandreg}
Let us start this section by introducing some important notation. The scalar product in $\re{R}^N$ and the Euclidean norm are denoted by $(\cdot ,\cdot)$ and $|\cdot|$, respectively. The duality pairing between a Banach space $V$ and its dual $V^*$ is represented by $\langle\cdot,\cdot\rangle_{V^*,V}$, and $\|\cdot\|_V$ stands for the norm of $V$. Given $1<p<\infty$, the conjugate exponent is denoted by $p'$. Further, the duality pairing between $L^p$ and $L^{p'}$ spaces is denoted by $\langle\cdot,\cdot\rangle_{p',p}$  (see \cite[Th. 4.11]{brezis}). We use the classical notation for the Sobolev space $W_0^{1,p}(\Omega)$ and the notation $W^{-1,p'}(\Omega)$ for its dual space.  Finally, we use the following bold notation $\mathbf{L}^p(\Omega):=L^p(\Omega)\times L^p(\Omega)$ for $1<p<\infty$.

This work is concerned with the numerical solution of the following class of variational inqualities of the second kind: find $u\in W_0^{1,p}(\Omega)$ such that
\begin{equation*}
\int_{\Omega}|\nabla u|^{p-2} (\nabla u,\nabla v)\,dx + g\int_{\Omega}|\nabla v|\,dx -g\int_{\Omega}|\nabla u|\,dx \geq \int_{\Omega} f (v-u)\,dx,\,\,\forall v\in W_0^{1,p}(\Omega),
\end{equation*}
where $1<p<\infty$, $g>0$ and $f\in L^{p'}(\Omega)$. 

It is well known that this variational inequality represents a necessary optimality condition for the following optimization problem of a non-smooth functional
\begin{equation}\label{eq:probvarin}
\underset{u\in W_0^{1,p}(\Omega)} {\min} J(u):=\frac{1}{p}\int_{\Omega}|\nabla u|^{p}\,dx + g\int_{\Omega}|\nabla u|\,dx - \int_{\Omega} f u\,dx.
\end{equation}
Therefore, we will focus on the numerical solution of \eqref{eq:probvarin}, by using optimization and variational techniques.

\begin{thm}\label{th:poborig}
Let $1<p<\infty$. Then, problem \eqref{eq:probvarin} has a unique solution $\overline{u}\in W_0^{1,p}(\Omega)$.
\end{thm}
\begin{proof}
Note that functional $J(\cdot)$ can be rewritten as
\[J(u)=\frac{1}{p}\|u\|_{W_0^{1,p}}^p + g\int_{\Omega}|\nabla u|\,dx - \int_{\Omega} f u\,dx.\]
Therefore, it is clear that $J(\cdot)$ is a continuous and strictly convex functional, which satisfies that
\[\underset{\|u\|_{W_0^{1,p}}\rightarrow\infty}{\lim} J(u)=+\infty.\]
This fact yields (see, for instance, \cite[Ch. 2]{jahn} and \cite[Ch. 1]{lions}) the existence of a unique solution $\overline{u}\in W_0^{1,p}(\Omega)$ for the problem \eqref{eq:probvarin}.\qed
\end{proof}

\subsection{A Multiplier Characterization}\label{sec:multiplier}
In this section, we use the Fenchel's duality theory to characterize the solution of problem \eqref{eq:probvarin} with a vectorial function, which acts as a multiplier. The aim of such a procedure is to obtain an optimality system which will be used to characterize the solutions of \eqref{eq:probvarin}.

For the sake of readability of the paper, let us briefly  describe the main ideas in Fenchel's theory. Let $V$ and $W$ be two Banach spaces with dual spaces $V^*$ and $W^*$, respectively. Let $\Lambda\in \mathcal{L}(V,W)$ be given and let $\Lambda^*\in \mathcal{L}(W^*,V^*)$ be its conjugate functional. Further,  let $\mathcal{F}:V\rightarrow\re{R}$ and $\mathcal{G}:W\rightarrow \re{R}$ be two given functionals. We are concerned with minimization problems in which the objective functional $J:V \rightarrow \re{R}$ can be decomposed as
\[
J(u):= \mathcal{F}(u)+\mathcal{G}(\Lambda u).
\]
In such a case, the problems we are interested in are given by
\begin{equation}\label{fenprimal}
\underset{u\in V}{\inf}\{\mathcal{F}(u)+ \mathcal{G}(\Lambda u)\}.
\end{equation}
Next, it is known that the associated dual problem of \eqref{fenprimal} is given by (see \cite[pp. 60--61]{ektem})
\begin{equation}\label{fendual}
\underset{\mathbf{q}\in W^*}{\sup}\{-\mathcal{F}^*(-\Lambda^* \mathbf{q})- \mathcal{G}^*(\mathbf{q})\}.
\end{equation}
Here, $\mathcal{F}^*:V^*\rightarrow \re{R}$ and $\mathcal{G}^*: W^*\rightarrow \re{R}$ denote the convex conjugate functionals of $\mathcal{F}$ and $\mathcal{G}$, respectively, \textit{i.e,},
\begin{equation*}
\mathcal{F}^*(-\Lambda^* \mathbf{q})=\sup_{v\in V}\left\{\langle -\Lambda^*\mathbf{q}\,,\, v\rangle_{V^*,V} - \mathcal{F}(v) \right\}\,\,\,\mbox{and}\,\,\,
\mathcal{G}^*(\mathbf{q})=\sup_{\mathbf{r}\in W^*} \left\{ \langle \mathbf{q}\,,\,\mathbf{r}\rangle_{W^*,W} - \mathcal{G}(\mathbf{r}) \right\}.
\end{equation*} 
Now, let us suppose that the primal problem \eqref{fenprimal} has a unique solution $\overline{u}\in V$ and that both $\mathcal{F}$ and $\mathcal{G}$ are convex and continuous. Then, \cite[Th. p. 59]{ektem} and \cite[Rem. 4.2, p. 60]{ektem} imply that no duality gap occurs, \textit{i.e.},
\[
\underset{u\in V}{\inf}\{\mathcal{F}(u)+ \mathcal{G}(\Lambda u)\}=\underset{\mathbf{q}\in W^*}{\sup}\{-\mathcal{F}^*(-\Lambda^* \mathbf{q})- \mathcal{G}^*(\mathbf{q})\},
\]
and, moreover, that the dual problem has at least one solution $\overline{\mathbf{q}}\in W^*$. 

Finally, Fenchel's duality theory allows us to characterize both the primal and dual solutions. Indeed, \cite[p. 61]{ektem} implies that $\overline{u}$ and $\overline{\mathbf{q}}$ satisfy the following system of equations
\begin{eqnarray}\label{fensys}
-\Lambda^*\overline{\mathbf{q}} &\in& \partial\mathcal{F}(\overline{u})\label{fensys1}\vspace{0.2cm}\\
\overline{\mathbf{q}}&\in& \partial\mathcal{G}(\nabla \overline{u}),\label{fensys2}
\end{eqnarray}
where $\partial \mathcal{F}(\overline{u})$ and $\partial \mathcal{G}(\nabla \overline{u})$ stand for the subdifferential of $\mathcal{F}$ at $\overline{u}$ and the subdifferential of $\partial \mathcal{G}$ at $\nabla \overline{u}$, respectively.

Let us turn our attention to problem \eqref{eq:probvarin}. First, we define $V:=W_0^{1,p}(\Omega)$, $V^*:=W^{-1,p'}(\Omega)$, $W=\mathbf{L}^{p}(\Omega)$, and we identify the dual space of $\mathbf{L}^p(\Omega)$ with $\mathbf{L}^{p'}(\Omega)$ (see \cite[Th. 4.11]{brezis}). 

Next, we introduce the functionals $\mathcal{F}:W_0^{1,p}(\Omega)\rightarrow \re{R}$ as $\mathcal{F}(u):=\frac{1}{p}\int_\Omega |\nabla u|^p\,dx + \int_\Omega fu\,dx$ and $\mathcal{G}: \mathbf{L}^p(\Omega)\rightarrow \re{R}$ as $\mathcal{G}(\mathbf{q}):=g \int_\Omega |\mathbf{q}|\,dx$. It can be easily verified that these two functionals are convex, continuous and proper. We also introduce the linear operator $\Lambda:W_0^{1,p}\rightarrow \mathbf{L}^p(\Omega)$  by $\Lambda u:= \nabla u$. Clearly, $\Lambda\in \mathcal{L}(W_0^{1,p},\mathbf{L}^p(\Omega))$. Thanks to these definitions, it is clear that problem \eqref{eq:probvarin} satisfy all the requirements of Fenchel's duality theory. Therefore, there exists at least one solution for the dual problemñ. Moreover, the solutions of primal and dual problems $\overline{u}\in W_0^{1,p}(\Omega)$ and $\overline{\mathbf{q}}\in \mathbf{L}^ {p'}(\Omega)$, respectively, satisfy the system \eqref{fensys1}-\eqref{fensys2}.

First, we study \eqref{fensys1}. In this case, since $\mathcal{F}$ is Gateaux differentiable, the subdifferential of $\mathcal{F}$ reduces to the Gateaux differential $\mathcal{F}'$ (see \cite[Prop. 5.3, p. 23]{ektem}). Therefore, \eqref{fensys1} implies that
$
\langle -\Lambda^*\overline{\mathbf{q}}\,,\, v\rangle_{W^{-1,p},W_0^{1,p}} =\langle\mathcal{F}'(\overline{u})\,,\, v\rangle_{W^{-1,p},W_0^{1,p}}, \,\forall v\in W_0^{1,p}(\Omega),
$
which is equivalent to
\[
-\langle \overline{\mathbf{q}}\,,\, \nabla v\rangle_{p',p} =\int_\Omega |\nabla \overline{u}|^{p-2} (\nabla \overline{u},\nabla v)\,dx -\int_\Omega f v\,dx, \,\forall v\in W_0^{1,p}(\Omega).
\]
Now, thanks to the Riesz's representation theorem in $L^p$ spaces (see \cite[Th. 4.11]{brezis}), there exists a unique $\mathbf{w}\in \mathbf{L}^{p'}(\Omega)$ such that $\langle \overline{\mathbf{q}}\,,\, \nabla v\rangle_{p',p} = \int_\Omega (\mathbf{w},\nabla v)\,dx$, which yields that
\begin{equation*}
-\int_\Omega (\mathbf{w},\nabla v)\,dx =\int_\Omega |\nabla \overline{u}|^{p-2} (\nabla \overline{u},\nabla v)\,dx -\int_\Omega f v\,dx=0, \,\forall v\in W_0^{1,p}(\Omega).
\end{equation*}

Next, we analyze \eqref{fensys2}. In this case, since $\mathcal{G}$ is not differentiable, \eqref{fensys2} implies that
\[
\mathcal{G}(\nabla\overline{u}) - \mathcal{G}(\mathbf{r}) \geq \langle \overline{\mathbf{q}}\,,\, \nabla \overline{u} -\mathbf{r}\rangle_{p',p},\,\,\forall \mathbf{r}\in \mathbf{L}^p(\Omega).
\]
By following similar argumentation as in \cite[Sec. 2.1]{dlRGpf}, we conclude that the last expression implies that
\begin{equation*}
\langle\overline{\mathbf{q}}\,,\,\mathbf{r}\rangle_{p',p}\leq g\int_\Omega|\mathbf{r}|\,dx, \forall \mathbf{r}\in \mathbf{L}^p(\Omega)\,\,\mbox{ and }\,\, g\int_\Omega |\nabla \overline{u}|\,dx =\langle\overline{\mathbf{q}}\,,\,\mathbf{r}\rangle_{p',p}.
\end{equation*}
Finally, thanks to \cite[Lem. 2.1]{dlRGpf} and \cite[pp. 85]{dlRGpf}, we obtain the following optimality system for \eqref{eq:probvarin}

\begin{subequations}\label{optsysnoreg}
\begin{equation}\label{optsysnoreg1}
\int_\Omega |\nabla \overline{u}|^{p-2} (\nabla \overline{u},\nabla v)\,dx+ \int_\Omega (\mathbf{w},\nabla v)\,dx  -\int_\Omega f v\,dx=0, \,\forall v\in W_0^{1,p}(\Omega),
\end{equation}

\begin{equation}\label{optsysnoreg2}
|\mathbf{w}(x)|\leq g, \mbox{ a.e. in $\Omega$},
\end{equation}

\begin{equation}\label{optsysnoreg3}
\left\{\begin{array}{lll}
\nabla \overline{u}(x) = 0&\mbox{or}\vspace{0.2cm}\\
\nabla \overline{u}(x) \neq 0 &\mbox{and}\,\, \mathbf{w}(x)=g\frac{\nabla \overline{u}(x)}{|\nabla \overline{u}(x)|}.
\end{array}\right.
\end{equation}
\end{subequations}

\begin{def1}
The active and inactive sets of the problem are defined by
\[
\mathcal{A}:=\{x\in\Omega\,:\, \nabla\overline{u}(x)\neq 0\}\,\mbox{ and }\,
\mathcal{I}:=\{x\in\Omega\,:\, \nabla\overline{u}(x)= 0\},
\]
respectively.
\end{def1}
\subsection{A Huber Regularization Procedure}\label{sec:Huber}
The non-differentiability of problem \eqref{eq:probvarin} can provoke instabilities in several numerical schemes, such as a primal-dual algorithm (see \cite{dlRGpf}). This issue can be appreciated in the fact that system \eqref{optsysnoreg} does not have a unique solution. Further, this lack of regularity prevents us from developing an algorithm based on optimization techniques, as proposed. A classical approach to this kind of problems is regularization. However, the question about what kind of regularization procedure is the most suitable is a hot topic (see \cite{dlRGpf,dlRG2D,Huilgol}). 

In this work, we propose a local regularization of Huber type. The big advantage of using such a procedure is that Huber regularization only changes locally the structure of the functional in \eqref{eq:probvarin}, preserving most of the qualitative properties of functional $J$.

Let us start by introducing, for $\gamma>0$, the function $\psi_\gamma:\re{R}^m\rightarrow\re{R}$ by
\begin{equation}\label{psi}
\psi_\gamma (z):=\left\{
\begin{array}{lll}
g|z|-\frac{g^2}{2\gamma} &\mbox{if $\gamma|z|\geq g$}\vspace{0.2cm}\\
\frac{\gamma}{2} |z|^2&\mbox{if $\gamma|z|< g$}.
\end{array}
\right.
\end{equation}
Note that $\psi_\gamma$ corresponds to a local regularization of the Euclidean norm. In Figure \ref{fig:huber} it is possible to appreciate the effect of this regularization in dimension one.
\begin{figure}
\begin{center}
\includegraphics[width=80mm, height=60mm]{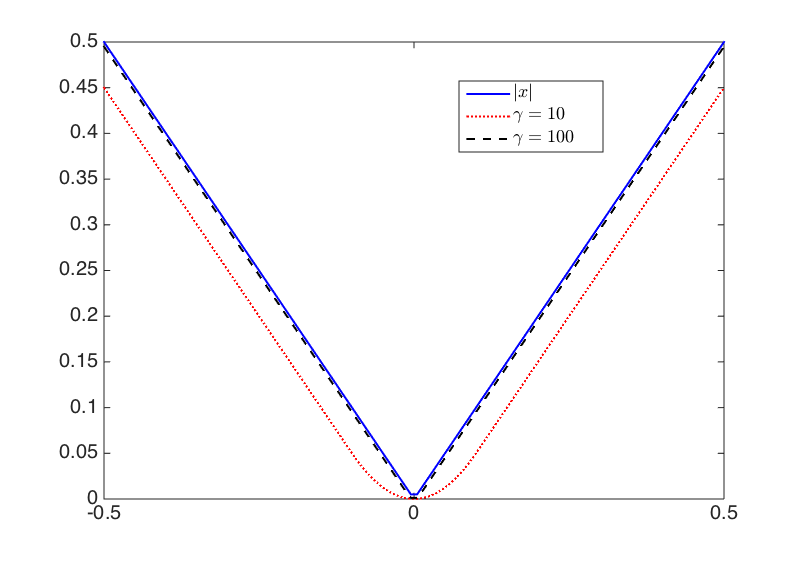}
\end{center}
\caption{Huber regularization in dimension 1.}\label{fig:huber}
\end{figure}

Next, by using the function $\psi_\gamma$, we propose the following regularized version of problem \eqref{eq:probvarin}

\begin{equation}\label{eq:probreg}
\min_{u\in W_0^{1,p}(\Omega)} J_\gamma(u):= \frac{1}{p}\int_\Omega |\nabla u|^p\, dx + \int_\Omega \psi_\gamma (\nabla u)\, dx - \int_\Omega f u\,dx.
\end{equation}
\begin{thm}
Let $1<p<\infty$ and $\gamma>0$. Then, problem \eqref{eq:probreg} has a unique solution $u_\gamma\in W_0^{1,p}(\Omega)$.
\end{thm}
\begin{proof}
First, let us state that function $\psi$ is a convex function \cite{dlRGpf}. Therefore, the functional $J_\gamma(u)$ has the same qualitative properties of functional $J(u)$. Consequently, the result follows in the same way as in Theorem \ref{th:poborig}.\qed
\end{proof}

We again propose the use of Fenchel's duality theory  to generate an optimality system for \eqref{eq:probreg}. Actually, in this case we only need to replace the functional $\mathcal{G}$ by the functional $\mathcal{G}_\gamma: \mathbf{L}^p(\Omega)\rightarrow \re{R}$ given by $\mathcal{G}_\gamma(\mathbf{p})=\int_\Omega \psi_\gamma(\mathbf{p})\,dx$, which is convex and continuous. Therefore, we can use the Fenchel's theory to state that the dual problem has at least one solution $\mathbf{q}_\gamma \in \mathbf{L}^{p'}(\Omega)$, and, moreover, that  $u_\gamma$ and $\mathbf{q}_\gamma$ satisfy the system \eqref{fensys1}-\eqref{fensys2}. 

Since the functional $\mathcal{F}$ has not changed, in this case \eqref{fensys1} reads as follows
\[
-\langle \mathbf{q}_\gamma\,,\, \nabla v\rangle_{p',p} =\int_\Omega |\nabla u_\gamma|^{p-2} (\nabla u_\gamma,\nabla v)\,dx -\int_\Omega f v\,dx, \,\forall v\in W_0^{1,p}(\Omega).
\]
Further, thanks to the Riesz's representation theorem in $L^p$ spaces (see \cite[Th. 4.11]{brezis}), there exists a unique $\mathbf{w}_\gamma\in \mathbf{L}^{p'}(\Omega)$ such that $\langle \mathbf{q}_\gamma\,,\, \nabla v\rangle_{p',p}=\int_\Omega (\mathbf{w}_\gamma,\nabla v)\,dx$. This fact yields that
\begin{equation*}
-\int_\Omega (\mathbf{w}_\gamma,\nabla v)\,dx =\int_\Omega |\nabla u_\gamma|^{p-2} (\nabla u_\gamma,\nabla v)\,dx -\int_\Omega f v\,dx=0, \,\forall v\in W_0^{1,p}(\Omega).
\end{equation*}
On the other hand, the functional $\mathcal{G}_\gamma$ is Gateaux differentiable. Therefore, in this case equation \eqref{fensys2} is given by
\[
\langle \mathbf{q}_\gamma\,,\,\mathbf{r}\rangle_{p',p} = \langle\mathcal{G}'_\gamma(\nabla u_\gamma)\,,\,\mathbf{r}\rangle_{p',p},
\]
which is equivalent to
\[
\langle \mathbf{q}_\gamma\,,\,\mathbf{r}\rangle_{p',p} = \int_\Omega \gamma g\dfrac{(\nabla u_\gamma, \mathbf{r})}{\max(g,\gamma |\nabla u_\gamma|)}\,dx,\,\,\forall \mathbf{r}\in\mathbf{L}^p(\Omega).
\]
Finally, since $\mathbf{w}_\gamma$ is the unique Riesz representative of $\mathbf{q}_\gamma$, we have that
\begin{equation}\label{wgradu}
\int_\Omega (\mathbf{w}_\gamma,\mathbf{r})\,dx=\int_\Omega \gamma g\dfrac{(\nabla u_\gamma, \mathbf{r})}{\max(g,\gamma |\nabla u_\gamma|)}\,dx,\,\,\forall \mathbf{r}\in\mathbf{L}^p(\Omega).
\end{equation}
Summarizing, we have the following regularized optimality system for \eqref{eq:probreg}.
\begin{subequations}\label{optsysreg}
\begin{equation}\label{optsysreg1}
\int_\Omega |\nabla u_\gamma|^{p-2} (\nabla u_\gamma,\nabla v)\,dx + \int_\Omega (\mathbf{w}_\gamma,\nabla v)\,dx-\int_\Omega f v\,dx=0, \,\forall v\in W_0^{1,p}(\Omega).
\end{equation}
\begin{equation}\label{optysreg2}
\mathbf{w}_\gamma(x)=g\gamma \dfrac{\nabla u_\gamma(x)}{\max(g,\gamma|\nabla u_\gamma(x)|)},\,\,\mbox{a.e. in $\Omega$ and $\gamma>0$}.
\end{equation}
\end{subequations}
\begin{def1}
The regularized active and inactive sets are given by
\[
\mathcal{A}_\gamma:=\{x\in\Omega\,:\, \gamma|\nabla u_\gamma(x)|\geq g\}\,\mbox{ and }\,\mathcal{I}_\gamma:=\{x\in\Omega\,:\, \gamma|\nabla u_\gamma(x)|< g\},
\]
respectively.
\end{def1}

\begin{lem}\label{lem:bound}
Let $1<p<\infty$ and $\gamma>0$. Then, the sequence of optima of \eqref{eq:probreg} is bounded in $W_0^{1,p}(\Omega)$.
\end{lem}
\begin{proof}
Let us start by noticing that
\[
\frac{1}{p}\int_\Omega |\nabla u_\gamma|^p\,dx - \int_\Omega fu_\gamma\,dx \leq J_\gamma(u)\leq J_\gamma(0)=0.
\]
Next, H\"older and Poincare inequalities imply the existence of a positive constant $C$, which only depends on $\Omega$ and $p$, such that
\[
\frac{1}{p}\|u_\gamma\|_{W_0^{1,p}}^p \leq C\| f\|_{L^{p'}} \|\nabla u_\gamma\|_{\mathbf{L}^p}.
\]
Since $p>1$, the last expression directly implies the result.\qed
\end{proof}

\begin{thm}
Let $1<p<\infty$. Then, the sequence $\{u_\gamma\}\subset W_0^{1,p}(\Omega)$ converges strongly in $W_0^{1,p}(\Omega)$ to the solution $\overline{u}$ of problem \eqref{eq:probvarin}.
\end{thm}
\begin{proof}
Note that $\overline{u}$ and $u_\gamma$ satisfy equations \eqref{optsysnoreg1} and \eqref{optsysreg1}, respectively. Thus, by subtracting \eqref{optsysreg1} from \eqref{optsysnoreg1}, we obtain that
\[\int_\Omega|\nabla \overline{u}|^p (\nabla\overline{u},\nabla v)\,dx - \int_\Omega |\nabla u_\gamma|^p (\nabla u_\gamma, \nabla v)\,dx + \int_\Omega (\mathbf{w},\nabla v)\,dx - \int_\Omega (\mathbf{w}_\gamma,\nabla v)\,dx=0,\,\,\forall v\in W_0^{1,p}(\Omega),\]
which, by choosing $v:=\overline{u}- u_\gamma$, yields that
\begin{equation}\label{u-ug}
\int_\Omega (|\nabla \overline{u}|^p \nabla \overline{u} - |\nabla u_\gamma|^p \nabla u_\gamma , \nabla (\overline{u}- u_\gamma))\,dx = \int_\Omega (\mathbf{w}_\gamma - \mathbf{w},\nabla (\overline{u}- u_\gamma))\,dx,\,\,\forall v\in W_0^{1,p}(\Omega).
\end{equation}
Next, by following \cite[Th. 2.5]{dlRGpf}, we establish the following pointwise bounds for $(\mathbf{w}_\gamma - \mathbf{w},\nabla (\overline{u}- u_\gamma))$ in the four disjoint sets: $\mathcal{A}\cap \mathcal{A}_\gamma$, $\mathcal{A}\cap \mathcal{I}_\gamma$, $\mathcal{A}_\gamma\cap\mathcal{I}$ and $\mathcal{I}_\gamma\cap \mathcal{I}$.

\begin{equation}\label{conv1234}
\begin{array}{lll}
\mathcal{A}\cap \mathcal{A}_\gamma &:&\,((\mathbf{w}_\gamma - \mathbf{w})(x),\nabla (\overline{u}- u_\gamma)(x)) \leq 0.\vspace{0.2cm}\\
\mathcal{A}\cap \mathcal{I}_\gamma &:&\,((\mathbf{w}_\gamma - \mathbf{w})(x),\nabla (\overline{u}- u_\gamma)(x))  < \gamma^{-1} g^2.\vspace{0.2cm}\\
\mathcal{A}_\gamma\cap\mathcal{I} &:&\,((\mathbf{w}_\gamma - \mathbf{w})(x),\nabla (\overline{u}- u_\gamma)(x)) <0.\vspace{0.2cm}\\
\mathcal{I}_\gamma\cap \mathcal{I} &:&\, ((\mathbf{w}_\gamma - \mathbf{w})(x),\nabla (\overline{u}- u_\gamma)(x)) < \gamma^{-1} g^2.
\end{array}
\end{equation}
Since, $\mathcal{A}\cap \mathcal{A}_\gamma$, $\mathcal{A}\cap \mathcal{I}_\gamma$, $\mathcal{A}_\gamma\cap \mathcal{I}$ and $\mathcal{I}_\gamma\cap \mathcal{I}$ provide a disjoint partitioning of $\Omega$, \eqref{u-ug} and the estimates in \eqref{conv1234} imply that
\begin{equation}\label{uug-gamma}
\int_\Omega (|\nabla \overline{u}|^p \nabla \overline{u} - |\nabla u_\gamma|^p \nabla u_\gamma , \nabla (\overline{u}- u_\gamma))\,dx < \int_\Omega \gamma^{-1}\,g^2\,dx.
\end{equation}
Next, we divide the proof in two cases: $p\geq 2$ and $1<p<2$.
\\\\
\noindent $p\geq 2$: In this case, \cite[Lem. 2.1]{Simon1} implies the existence of a positive constant $C_p$, depending on $p$, such that
\[
(|\nabla \overline{u}(x)|^p \nabla \overline{u}(x) - |\nabla u_\gamma (x)|^p \nabla u_\gamma (x), \nabla (\overline{u}- u_\gamma)(x))\geq C_p |\nabla \overline{u}(x)- \nabla u_\gamma(x)|^p,\,\,\mbox{a.e. in $\Omega$}.
\]
Therefore, by plugging the inequality above in \eqref{uug-gamma}, we have that
\[
C_p\int_\Omega |\nabla \overline{u} -\nabla u_\gamma|^p\,dx < \int_\Omega \gamma^{-1}\,g^2\,dx,
\]
which implies that
\begin{equation}\label{ugup>2}
\|\overline{u}-u_\gamma\|_{W_0^{1,p}} < \left(\dfrac{g^2 meas(\Omega)}{C_p\gamma}\right)^{\frac{1}{p}},
\end{equation}
Finally, since $\Omega$ is bounded, \eqref{ugup>2} allows us to conclude that $u_\gamma\rightarrow \overline{u}$ strongly in $W_0^{1,p}(\Omega)$, for $p\geq2$.
\vspace{0.2cm}

\noindent $1<p<2$: In this case  \cite[Lem. 2.1]{Simon1} implies the existence of a positive constant $D_p$, depending on $p$, such that
\[
(|\nabla \overline{u}(x)|^p \nabla \overline{u}(x) - |\nabla u_\gamma (x)|^p \nabla u_\gamma (x), \nabla (\overline{u}- u_\gamma)(x))\geq D_p \dfrac{|\nabla \overline{u}(x)- \nabla u_\gamma(x)|^2}{(|\nabla\overline{u}(x)| + |\nabla u_\gamma (x)|)^{2-p}},\,\,\mbox{a.e. in $\Omega$}.
\]
Thus, if we consider this inequality in \eqref{uug-gamma}, we have that
\begin{equation}\label{1<p<21}
D_p \int_\Omega\dfrac{|\nabla \overline{u}-\nabla u_\gamma|^2}{(|\nabla \overline{u}| + |\nabla u_\gamma|)^{2-p}}\, dx < \int_\Omega \gamma^{-1}\, g^2\,dx.
\end{equation}
On the other hand, note that H\"older's inequality implies that
\[
\begin{array}{lll}
\int_\Omega |\nabla (\overline{u} - u_\gamma)|^p\,dx &=& \int_\Omega\dfrac{ |\nabla (\overline{u} - u_\gamma)|^p}{(|\nabla \overline{u}| + |\nabla u_\gamma|)^{\frac{p(2-p)}{2}}}(|\nabla \overline{u}| + |\nabla u_\gamma|)^{\frac{p(2-p)}{2}}\,dx\vspace{0.2cm}\\&\leq&\left[\int_\Omega \dfrac{|\nabla (\overline{u}-u_\gamma)|^2}{(|\nabla \overline{u}| + |\nabla u_\gamma|)^{2-p}}\,dx\right]^{\frac{p}{2}} \left[\int_\Omega (|\nabla \overline{u}| + |\nabla u_\gamma|)^p\,dx\right]^{\frac{2-p}{2}}.
\end{array}
\]
This last inequality and \eqref{1<p<21} yield that
\[
\left(\int_\Omega |\nabla (\overline{u}-u_\gamma)|^p\,dx\right)^{\frac{1}{p}} \leq \left(\frac{g^2 meas(\Omega)}{D_p \gamma}\right)^{\frac{1}{p}}\left[\int_\Omega (|\nabla \overline{u}| + |\nabla u_\gamma|)^p\,dx\right]^{\frac{2-p}{2p}},
\]
which implies that
\[
\|\overline{u}-u_\gamma\|_{W_0^{1,p}} \leq \left(\frac{g^2 meas(\Omega)}{D_p \gamma}\right)^{\frac{1}{p}}\left[\| \overline{u}\|_{W_0^{1,p}} + \|u_\gamma\|_{W_0^{1,p}}\right]^{\frac{2-p}{2p}}.
\]
Finally, since the sequence $\{u_\gamma\}$ is bounded in $W_0^{1,p}(\Omega)$ (see Lemma \ref{lem:bound}), there exists a positive constant $\widetilde{D}_p$ such that
\begin{equation}\label{ugup<2}
\|\overline{u}-u_\gamma\|_{W_0^{1,p}} \leq \frac{\widetilde{D}_p meas(\Omega)}{\gamma^{\frac{1}{p}}},
\end{equation}
which, since $\Omega$ is bounded, allows us to conclude that $u_\gamma\rightarrow \overline{u}$ strongly in $W_0^{1,p}(\Omega)$.\qed
\end{proof}

\section{Preconditioned Descent Algorithms}\label{sec:preconalgo}
In this section we analyze the application of descent algorithms for solving the regularized problem \eqref{eq:probreg}. We divide this study in two cases: $1<p<2$ and $p\geq 2$. Thus, we need to consider all the particular issues that arise in these two scenarios, such as existence of admissible descent directions.

Descent methods work by finding, at the current iterate $u_k\in V$, a search direction $w_k\in V$ such that  $J_\gamma(u_k + t w_k)$ is decreasing at $t=0$, \textit{i.e.}, such that
\[
\langle J_\gamma'(u_k)\,,\, w_k\rangle_{V^*,V}<0.
\]
Here, $V$ stands for a Banach space and $V^*$ for its dual space. Although this kind of algorithms are usually suitable for differentiable problems, several issues arise. Mainly, the descent provoked in the function can be very small. This problem usually appears when the contour maps of the functional are very prolonged near the minimizer. Further, in the particular case of problem \eqref{eq:probreg}, since this problem involve the $p$-Laplacian operator, the difficulties associated to this structure need to be taken into account (see \cite{Huang}). 

An innovative idea to deal with these issues is to use a suitable preconditioner in the computation of the search direction.  In \cite{Huang}, the authors successfully  implement this idea in a finite dimension setting for the $p$-Laplacian problem. Here, we propose and analyze a similar approach in function spaces, for the regularized problem \eqref{eq:probreg}. In fact, we determine the search direction $w_k$ by solving the following equation
\[
P_k(w_k,v)=-\langle J_\gamma'(u_k)\,,\,v\rangle_{V^*,V},\,\,\forall v\in V,
\]
where $V$ is a suitable Banach space and the form $P_k:V\times V\rightarrow \re{R}$ is chosen as a variational approximation of the $p$-Laplacian operator. 

By taking into account the last discussion, we obtain the following general algorithm.

\begin{algorithm}\label{algo-gen}
Initialize $u_0\in V$ and set $k=0$.  

For $k=1,2,\ldots$ do
\begin{enumerate}
\item If $J_\gamma'(u_k)=0$, STOP.
\item Solve $P_k(w_k,v)=-\langle J_\gamma'(u_k)\,,\,v\rangle_{V^*,V},\,\,\forall v\in V$, for a descent direction $w_k$.
\item Perform a line search algorithm to determine the step size $\alpha_k$.
\item Update $u_{k+1}:= u_k + \alpha_k w_k$ and set $k=k+1$.
\end{enumerate}
\end{algorithm}
Several issues arise when discussing the convergence properties of this algorithm. By considering the discussion in \cite[Sec. 2.2.1]{hinetal}, global convergence of this algorithm depends on the admissibility of $w_k$ and $\alpha_k$. 


Admissibility of search directions $w_k$ depends on the way in which we define $P_k$. Thus, since the behaviour of $J_\gamma$ depends on the value of $p$, existence and admissibility of descent directions will be discuss in the next sections considering the two cases $1<p<2$ and $p>2$, separately. 

On the other hand, the line search strategy in step 3 of Algorithm \ref{algo-gen} can be performed in several ways. Exact line search algorithms, \textit{i.e.}, algorithms which find $\alpha_k$ such that
\[
J_\gamma(u_k+\alpha_k w_k)= \min_{\alpha\geq 0}J_\gamma(u_k+\alpha w_k),
\]
are known to be expensive, specially when the iterate is far from the solution \cite{SunYuan}.  Therefore, we will use inexact line search techniques. Further, in order to proof convergence for descent algorithms like \ref{algo-gen}, these inexact techniques need to be efficient, according to the following definition.
\begin{def1}
A line search strategy is called efficient if there exists a constant $\zeta > 0$, independent of $u_k$ and $w_k$, such that 
\[
J_\gamma (u_k+ \alpha_k)\leq J_\gamma(u_k) - \zeta \left(\frac{\langle J_\gamma'(u_k)\,,\, w_k\rangle_{W^{-1,p'}, W_0^{1,p}}}{\|w_k\|_{W_0^{1,p}}}\right)^2.
\]
\end{def1}

A classical line search strategy is the so called Wolfe-Powell rule. This method consists in accepting a positive steplength $\alpha_k$ if
\begin{subequations}\label{eq:Wolfe}
\begin{equation}\label{eq:Wolfe1}
J_\gamma(u_k+\alpha w_k)\leq J_\gamma(u_k) + \sigma_1 \alpha_k \langle J'_\gamma(u_k)\,,\, w_k\rangle_{V^*,V},
\end{equation}
\begin{equation}\label{eq:Wolfe2}
\langle J'_\gamma(u_k+\alpha w_k)\,,\, w_k\rangle_{V^*,V}\geq \sigma_2\langle J'_\gamma(u_k)\,,\, w_k\rangle_{V^*,V},
\end{equation}
\end{subequations}
where $0<\sigma_1<\sigma_2<1$. Wolfe-Powell rule is known to satisfy the previous efficiency requirements and it will be used as a central requirement in the coming convergence results.

\subsection{The $1<p<2$ case}
In this section, we construct an algorithm, based on Algorithm \ref{algo-gen}, for the problem \eqref{eq:probreg}, when $1<p<2$. Due to the structure of the problem, we will analyze this case in function spaces. Therefore, we discuss the space $V$ in which the algorithm is constructed, define the bilinear form $P_k(\cdot,\cdot)$,  analyze the equation $P_k(w,v)=-\langle J_\gamma(u)\,,\,v\rangle_{V^*,V}$, and, finally, we write the algorithm and prove a global convergence result.

\begin{def1}
Let $1<p<2$, $\epsilon>0$ and $u\in W_0^{1,p}(\Omega)$. We define $H^u_0(\Omega)$ as the completion $\mathcal{D}(\Omega)$ with respect to the norm
\[\|z\|_{H_0^u}=\left(\int_\Omega (\epsilon + |\nabla u|)^{p-2} |\nabla z|^2\,dx \right)^{\frac{1}{2}}.\]
\end{def1}
\begin{thm}\label{teo:hilbert}
Let $1<p<2$, $\epsilon>0$ and $u\in W_0^{1,p}(\Omega)$. Then, $H^u_0(\Omega)$ is a Hilbert space with the inner product
\begin{equation}\label{prodHu}
(z,w)_{H_0^u}=\int_\Omega (\epsilon + |\nabla u|)^{p-2} (\nabla z,\nabla w)\,dx.
\end{equation}
Furthermore, the following inclusion holds, with continuous injections
\begin{equation}\label{spacesp<2}
H_0^1(\Omega)\subset H_0^u(\Omega)\subset W_0^{1,p}(\Omega).
\end{equation}
\end{thm}
\begin{proof}
Let us start by pointing out that \eqref{prodHu} is a positive definite  bilinear form, which fits the structure analyzed in \cite[p. 214]{coffman} and \cite[pp. 268-269]{Trudinger}. 

Next, we analyze the coefficient $(\epsilon + |\nabla u|)^{p-2}$.  First, note that
\[
(\epsilon + |\nabla u(x)|)^{p-2} = \dfrac{1}{(\epsilon + |\nabla u(x)|)^{2-p}},\,\,\mbox{a.e. in $\Omega$},
\]
which implies, since $2-p>0$, that
\[
\dfrac{1}{(\epsilon + |\nabla u(x)|)^{2-p}}\leq \dfrac{1}{\epsilon^{2-p}},\,\,\mbox{a.e. in $\Omega$}.
\]
The last two expressions yield that
\begin{equation}\label{coef1}
(\epsilon + |\nabla u|)^{p-2}\in L^\infty(\Omega)\subset L^1(\Omega).
\end{equation}
Now, note that
\[
\left[(\epsilon + |\nabla u(x)|)^{p-2} \right]^{-1}= (\epsilon + |\nabla u(x)|)^{2-p},\,\,\mbox{a.e. in $\Omega$}.
\]
Since $u\in W_0^{1,p}(\Omega)$ and $2-p<p$, we  can state that
\begin{equation}\label{coef2}
(\epsilon + |\nabla u|)^{2-p}\in L^1(\Omega).
\end{equation}
Consequently, \eqref{coef1}, \eqref{coef2}, \cite[Lem. 3.3]{coffman} and \cite[p. 268-269]{Trudinger}, yield that the Hilbert space $H^u_0(\Omega)$ is well defined.

We now prove \eqref{spacesp<2}. Let $z\in H_0^1(\Omega)$. First, note that, thanks to \eqref{coef1}, there exists a positive constant $C_1>0$, such that
\[
\int_\Omega (\epsilon + |\nabla u|)^{p-2} |\nabla z|^2\,dx \leq C_1 \int_\Omega |\nabla z|^2\, dx,
\]
which implies the existence of a positive constant $\widetilde{C}_1$ such that
\begin{equation}\label{incspa1}
\|z\|_{H_0^u}\leq \widetilde{C}_1 \|z\|_{H_0^1}.
\end{equation}
Further, let $z\in H_0^u(\Omega)$. H\"older's inequality implies that
\[
\begin{array}{lll}
\int_\Omega |\nabla z|^p\,dx &=& \int_\Omega \dfrac{|\nabla z|^p}{\left(\epsilon + |\nabla u|\right)^{\frac{p(2-p)}{2}}} \left(\epsilon + |\nabla u|\right)^{\frac{p(2-p)}{2}}\,dx\vspace{0.2cm}\\&\leq &\left[\int_\Omega \dfrac{|\nabla z|^2}{(\epsilon + |\nabla u|)^{2-p}}\,dx\right]^{\frac{p}{2}} \left[\int_\Omega (\epsilon + |\nabla u|)^p\,dx\right]^{\frac{2-p}{2p}}.
\end{array}
\]
Next, since $u\in W_0^{1,p}(\Omega)$, the last expression implies the existence of a positive constant $C_2$ such that
\[
\int_\Omega |\nabla z|^p\,dx \leq C_2 \left[\int_\Omega (\epsilon + |\nabla u|)^{p-2}|\nabla z|^2\,dx\right]^{\frac{p}{2}},
\]
which implies the existence of a positive constant $\widetilde{C}_2$ such that
\begin{equation}\label{incspa2}
\|z\|_{W_0^{1,p}}\leq \widetilde{C}_2 \|z\|_{H_0^u}.
\end{equation}
Summarizing, \eqref{incspa1} and \eqref{incspa2} imply that
\[
H_0^1(\Omega)\subset H_0^u(\Omega) \subset W_0^{1,p}(\Omega), 
\]
with continuous injections.\qed
\end{proof}

We propose our algorithm, considering that $V:=H_0^{\hat{u}}(\Omega)$, for some suitable $\hat{u}\in W_0^{1,p}(\Omega)$. Moreover, it looks natural that the form $P_k$ will be defined as follows.
\begin{equation*}
P_k(w,v):=\int_\Omega(\epsilon + |\nabla \hat{u}|)^{p-2} (\nabla w,\nabla v)\,dx.
\end{equation*}
Here the small parameter $\epsilon>0$ helps the algorithm to handle possible degeneracy when $\nabla \hat{u}=0$. Note that $P_k$ is a linearization of the weak form $\int_\Omega |\nabla \hat{u}|^{p-2} (\nabla \hat{u},\nabla v)$. 

Next, note that $H_0^{\hat{u}}(\Omega)\subset W_0^{1,p}(\Omega)$, for all $\hat{u}\in W_0^{1,p}(\Omega)$. Next, we define by $\widehat{J}'_\gamma(\hat{u})$ the restriction of $J'_\gamma(\hat{u})$ to $H_0^{\hat{u}}(\Omega)$. Therefore, we can state that $\widehat{J}'_\gamma(\hat{u})\in H_0^{\hat{u}}(\Omega)^*$ and that
\begin{equation}\label{eq:dualeseq}
\langle \widehat{J}'_\gamma(u)\,,\,v\rangle_{{H_0^{\hat{u}}}^*,H_0^{\hat{u}}}= \langle J_\gamma'(u)\,,\,v\rangle_{W^{-1,p'},W_0^{1,p}},\,\,\forall v\in H_0^{\hat{u}}(\Omega).
\end{equation}
For further details, we refer the reader to \cite[Rem. 3, p. 136]{brezis}.

Summarizing, we need to analyze the following variational equation
\begin{equation}\label{Pk<2}
\int_\Omega(\epsilon + |\nabla \hat{u}|)^{p-2} (\nabla w,\nabla v)\,dx =-\langle \widehat{J}'_\gamma(\hat{u})\,,\,v\rangle_{{H_0^{\hat{u}}}^*,H_0^u}, \,\,\forall v\in H_0^{\hat{u}}(\Omega).
\end{equation}
It is clear that a solution for a similar equation will play the role of the descent direction in our Algorithm. Therefore, we need to prove that this equation has, at least, one solution in $H_0^{\hat{u}}(\Omega)$.

This existence result is a direct consequence of the Riesz-Fr\'echet representation theorem (see \cite[Th. 5.5]{brezis}). In fact, we know that $H_0^{\hat{u}}(\Omega)\subset W_0^{1,p}(\Omega)$ is a Hilbert space for $1<p<2$. Moreover, we know that $\int_\Omega(\epsilon + |\nabla \hat{u}|)^{p-2} (\nabla w,\nabla v)\,dx$ is the scalar product of this Hilbert space. Consequently, the Riesz-Fr\'echet representation theorem implies the existence of a unique $w\in H_0^{\hat{u}}(\Omega)$ such that 
\[
\int_\Omega (\epsilon + |\nabla \hat{u}|)^{p-2} (\nabla w,\nabla v)\, dx = -\langle \widehat{J}'_\gamma(\hat{u})\,,\,v\rangle_{{H_0^{\hat{u}}}^*,H_0^u},\,\,\,\forall v\in H_0^{\hat{u}}(\Omega).
\]

Summarizing, the Algorithm \ref{algo-gen} takes the following form for $1<p<2$.

\begin{algorithm}\label{algo1<p<2}
Initialize $u_0\in W_0^{1,p}(\Omega)$ and set $k=0$.  

For $k=1,2,\ldots$ do
\begin{enumerate}
\item If $J'_\gamma(u_k)=0$, STOP.
\item Find a descent direction $w_k\in H^{u_k}_0(\Omega)$ by solving the following variational equation
\begin{equation}\label{eq:wp<2algo}
\begin{array}{lll}
\int_\Omega (\epsilon + |\nabla u_k|)^{p-2}(\nabla w_k,\nabla v)\,dx =-\langle \widehat{J}'_\gamma(u_k)\,,\, v\rangle_{H_0^{u_k*},H_0^{u_k}} \vspace{0.2cm}\\= -\int_\Omega |\nabla u_k|^{p-2}(\nabla u_k,\nabla v)\,dx -g\gamma \int_\Omega \frac{(\nabla u_k,\nabla v)}{\max(g,\gamma|\nabla u_k|)}\, dx +\int_\Omega f\,v\,dx,\,\,\forall v\in H_0^{u_k}(\Omega).
\end{array}
\end{equation} 

\item Perform an efficient line search technique to obtain $\alpha_k$.
\item Update $u_{k+1}:= u_k + \alpha_k w_k\in W_0^{1,p}(\Omega)$ and set $k=k+1$.
\end{enumerate}
\end{algorithm}

Clearly, the equation \eqref{eq:wp<2algo} has a unique solution $w_k\in H_0^{u_k}(\Omega) \subset W_0^{1,p}(\Omega)$, for all $k\in\re{N}$. Thus, Algorithm \ref{algo1<p<2} is well defined. However, it is mandatory to prove that $w_k\in W_0^{1,p}(\Omega)$ is, indeed, an admissible descent direction. First, we prove that $w_k$ is a descent direction. In fact, note that from \eqref{eq:dualeseq} and \eqref{eq:wp<2algo}, we obtain that
\begin{equation*}
\begin{array}{lll}
-\langle J_\gamma'(u)\,,\,w_k\rangle_{W^{-1,p'},W_0^{1,p}}&=& -\langle \widehat{J}'_\gamma(u)\,,\,w_k\rangle_{{H_0^{\hat{u}}}^*,H_0^{\hat{u}}}\vspace{0.2cm}\\&= &\int_\Omega (\epsilon + |\nabla u_k|)^{p-2} |\nabla w_k|^2\,dx\vspace{0.2cm} \\&=& \|w_k\|_{H_0^{u_k}}^2,
\end{array}
\end{equation*}
which yields that
\begin{equation}\label{eq:wdes}
\langle J_\gamma'(u)\,,\,w_k\rangle_{W^{-1,p'},W_0^{1,p}}<0.
\end{equation}

Next, let us discuss the admissibility of $w_k$. Note that if we had defined $P_k$ as the variational version of the Laplacian operator, \textit{i.e.}, $P_k(u,v)=\int_\Omega (\nabla u\,,\,\nabla v)\,dx$, the sequence generated by the associated version of the Algorithm \ref{algo1<p<2} would be such that $\{u_k\}\subset H^1_0(\Omega)\subset W_0^{1,p}(\Omega)$. In this case, it is possible to state the existence of $q<2$ such that $H_0^1(\Omega)\subset W_0^{1,q}(\Omega)\subset W_0^{1,p}(\Omega)$ (see \cite{brezis,Triebel}). On the other hand, note that the sequence $\{u_\ell\}$ generated by Algorithm \ref{algo1<p<2} yields that $\{u_\ell\}\in \cup_{j\in \re{N}} H_0^{u_j}(\Omega)\subset W_0^{1,p}(\Omega)$.  These arguments suggest that, using interpolation theory \cite{Triebel}, a similar inclusion result can be obtain for $P_k(w,v):=\int_\Omega(\epsilon + |\nabla \hat{u}|)^{p-2} (\nabla w,\nabla v)\,dx$. Thus, we make the following assumption.

\begin{assu}\label{assu:q}
There exists $q$, $1<p<q<2$, such that $\cup_{j\in \re{N}} H_0^{u_j}(\Omega)\subset W_0^{1,q}(\Omega)\subset W_0^{1,p}(\Omega)$.
\end{assu}

\begin{prop}\label{th:zout}
Let $\{u_k\}$ be the sequence generated by Algorithm \ref{algo1<p<2} and suppose that the step length $\alpha_k$ satisfies the Wolfe-Powell conditions \eqref{eq:Wolfe}. Furthermore, let us suppose that the Assumption \ref{assu:q} holds. Then, the Zoutendijk condition is verified, \textit{i.e.},
\begin{equation}\label{eq:zout}
\sum_{k=0}^{\infty}\cos^2 \phi_k= \infty,
\end{equation}
\end{prop}
where $\cos \phi_k =-\frac{\langle J_\gamma'(u_k)\,,\, w_k\rangle_{W^{-1,p'}, W_0^{1,p}}}{\|J'_\gamma(u_k)\|_{W^{-1,p'}}\|w_k\|_{W_0^{1,p}}}$.
\begin{proof}
First, note that Theorem \ref{th:poborig} implies that the functional $J_\gamma$ is bounded below in $W_0^{1,p}(\Omega)$. Next, let us recall that the functional $J_\gamma$ can be written as
\[
J_\gamma(u)= \mathcal{F}(u) + \mathcal{G}_\gamma(\nabla u),
\]
where $\mathcal{F}(u)$ and $\mathcal{G}_\gamma(\nabla u)$ are given in Section \ref{sec:Huber}. It was previously stated that both $\mathcal{F}$ and $\mathcal{G}_\gamma$ are continuously differentiable in $W_0^{1,p}(\Omega)$. Moreover, it is known that $\mathcal{F}$ is actually twice differentiable, since this functional represents the variational version of the Dirichlet problem for the $p$-Laplacian operator (see \cite{bermejomg,glomaro}). Thus, $\mathcal{F}$ has a Lipschitz continuous gradient in $W^{-1,p'}(\Omega)$. On the other hand, in Section \ref{sec:Huber} we stated that
\[
\langle \mathcal{G}'_\gamma(\nabla u)\,,\, v\rangle_{W^{-1,p'}, W_0^{1,p}}= g\,\gamma\int_\Omega \frac{(\nabla u\,,\,\nabla v)}{\max (g,\gamma|\nabla u|)}\,dx.
\]
Next, thanks to the Assumption \ref{assu:q}, the max function involved in the last expression is slantly differentiable (see \cite{hik1}). Consequently, we can state that $\mathcal{G}'_\gamma(\nabla u)$ is slantly differentiable in $W_0^{1,p}(\Omega)$. Therefore, thanks to \cite[Th. 2.6, pp. 1205]{chenqi}, $\mathcal{G}'_\gamma$ is Lipschitz continuous in $W^{-1,p'}(\Omega)$.

Summarizing, we know that $J_\gamma$ is bounded below and continuously differentiable in $ W_0^{1,p}(\Omega)$, and its gradient is Lipschitz continuous in $W^{-1,p'}(\Omega)$. Therefore, since we assume that $\alpha_k$ satisfies the Wolfe-Powell conditions, all the hypothesis of Zoutendijk theorem are satisfied (see, for instance, \cite[pp. 29]{kanzow} and \cite[Lem. 2.5.6]{SunYuan} and the references therein ), and, consequently \eqref{eq:zout} holds.\qed
\end{proof}

\begin{thm}\label{th:convp<2}
Let $\{u_k\}$ be the sequence generated by Algorithm \ref{algo1<p<2} and suppose that the step length $\alpha_k$ satisfies the Wolfe-Powell conditions \eqref{eq:Wolfe}. Furthermore, let us suppose that the Assumption \ref{assu:q} holds. Then, the sequence $\{u_k\}$ converges to the uniquely determined global minimum of $J_\gamma$.
\end{thm}
\begin{proof}
First, note that the Hanner's inequality \cite{liebloss} and the convexity of function $\psi$ imply that $J_\gamma$ is a uniformly convex functional in $W_0^{1,p}(\Omega)$. Further, Proposition \ref{th:zout}  guarantees that the  Zoutendijk condition holds. Therefore, the result directly follows from \cite[Th. 4.7]{kanzow}.\qed
\end{proof}

\subsection{The $p>2$ case}\label{sec:algop>2}
In this section, we construct an algorithm, based on Algorithm \ref{algo-gen}, for a discrete approximation of the problem \eqref{eq:probreg}, when $p\geq 2$. Our first aim was to construct an algorithm in function spaces. However, the structure of the problem prevents us from this goal. Particularly, there are regularity issues regarding the search direction. Indeed, we have the following result.
 
\begin{thm}\label{th:existencep>2}
Let $p\geq 2$ and $\varphi\in W^{-1,p'}(\Omega)$. Then, the variational equation
\begin{equation}\label{eq:wp>2}
\int_\Omega (\nabla w,\nabla v)\,dx = \langle \varphi\,,\, v\rangle_{W^{-1,p'},W_0^{1,p}},\,\,\forall v\in W_0^{1,p}(\Omega)
\end{equation}
has a unique solution $w\in W_0^{1,p'}(\Omega)$. Furthermore, there exists $K>0$ such that
\begin{equation}\label{equivp>2}
K \|w\|_{W_0^{1,p}}\leq \|\varphi\|_{W^{-1,p'}}\leq \|w\|_{W_0^{1,p}}.
\end{equation}
\end{thm}
\begin{proof}
Since $\Omega\subset \re{R}^2$ is assumed to be a bounded domain with regular boundary, \cite[Th. 4.6]{Simader} immediately implies the result.\qed
\end{proof}
Note that $J_\gamma:W_0^{1,p}(\Omega)\rightarrow\re{R}$, which implies that $J_\gamma'(u)\in W^{-1,p'}(\Omega)$, for all $u\in W_0^{1,p}(\Omega)$.  Therefore, it is possible to find a unique solution $\widehat{w}_k$ for the following equation
\[\int_\Omega (\nabla \widehat{w}_k,\nabla v)\,dx =-\langle J'_\gamma(u_k)\,,\, v\rangle_{W^{-1,p'},W_0^{1,p}},\,\,\forall v\in W_0^{1,p}(\Omega).\]
However, $\widehat{w}_k\in W_0^{1,p'}(\Omega)\supset W_0^{1,p}(\Omega)$ for $p\geq 2$. This fact prevents us from directly constructing an algorithm like Algorithm \ref{algo-gen}, since $u_{k+1}=u_k + \alpha_k \widehat{w}_k \in W_0^{1,p'}(\Omega)$. Moreover, Theorem \ref{th:existencep>2} can be extended to more general elliptic forms than the Laplacian. These results can be found in, \textit{e.g.}, \cite{groeger}. Consequently, the regularity issue prevails, for several elliptic choices for $P_k$.

A possible solution for this issue is to pose the problem in a suitable $H^s(\Omega)$ space, with $s\in \re{R}$ such that $H^s(\Omega)\subset W_0^{1,p}(\Omega)$. Indeed, it is known that for $p>2$ and $\hat{u}\in W_0^{1,p}(\Omega)$, the following inclusions hold, with continuous injections (see \cite{casas})
\begin{equation}\label{eq:inmp>2a}
W_0^{1,p}(\Omega)\subset H_0^{\hat{u}}(\Omega)\subset H_0^1(\Omega).
\end{equation}
Furthermore, it is possible to state that (see \cite[Prop. 1 pp. 96]{daulions})
\begin{equation}\label{eq:inmp>2b}
H^s(\Omega)\subset H^1(\Omega), \,\,\forall s>1.
\end{equation}
Thus, \cite[Rem. 2 pp. 96]{daulions}, \eqref{eq:inmp>2a} and \eqref{eq:inmp>2b} yield the existence of a $\hat{s}\in \re{R}$ such that
\[
H^{\hat{s}}(\Omega)\subset W_0^{1,p}(\Omega)\subset H^1(\Omega).
\]
Therefore, we can define $P_k$ as the scalar product in $H^{\hat{s}}(\Omega)$. However, several technical challenges arise with this idea. For instance, the actual value of $\hat{s}$ is unknown, and the numerical realisation of the search direction requires the implementation of the Fourier transform of several functions. We consider that all of these issues are beyond the scope of this paper, and will be considered in a future contribution.

Another possible idea to overcome the regularity problem is given by a smoothing step.
\[
W_0^{1,p'}(\Omega) \ni \widehat{w}_k \mapsto w_k \in W_0^{1,p}(\Omega).
\]
In \cite[Sec. 6]{ulbrich}, the author discusses the definition and properties of such a procedure. Though this smoothing procedures are designed for fixing regularity issues in function spaces like the one we have is this paper, they need several technical assumptions. These assumptions, at least in this context, can be very restricitve and can even reduce the admissible set of solutions for equation $P_k(w,v)=-\langle J'_\gamma(u)\,,\,v\rangle$ to the empty set. On the other hand, it is known that in finite dimensional spaces no smoothing step is needed, so we can define $\widehat{w}_k$ as the search direction for the descent algorithm (see \cite[Sec. 6.1]{ulbrich}).  

By taking into account the argumentation above, we consider that the best solution is to analyze the problem with a ``discretize then optimize'' approach. Thus, we propose a finite element discretization of the problem \eqref{eq:probreg}. Next, we propose and study a preconditioned algorithm for the case $p>2$ in finite dimension spaces.

We propose a discretization with first order finite elements, following ideas in \cite{barret,glomaro}. Thus, let $T^h$ be a regular triangulation, in the sense of Ciarlet, of $\Omega$. Next, let $\Omega^h$ be a polygonal approximation to $\Omega$, given by $\Omega^h=\bigcup_{\tau\in T^h} \overline{\tau}$, where all the open disjoint regular triangles $\tau$ have maximum diameter bounded by $h$. Further, for any two triangles, their closures are either disjoint or have a common vertex or a common side. Finally, let $\{P_j\}_{j=1,\ldots,N}$ be the vertices associated with the triangulation $T^h$. Hereafter,  we assume that $P_j\in \partial\Omega^h$ implies that $P_j\in \partial\Omega$ and that $\Omega^h\subset \Omega$. In this paper we will only consider first order approximation, because of the limited higher order regularity for the solutions of the $p$-Laplacian (see \cite{Huang} and the references therein). Taking the above discussion into account, we introduce the following finite-dimensional spaces associated with the triangulation $T^h$
\[
W_0^h:=\{v\in C(\overline{\Omega^h})\,:\, \mbox{ $v|_{\tau}\in \re{P}_1, \forall \tau \in T^h$ and $v=0$ on $\partial \Omega^h$}  \},
\]
where $\re{P}_1$ is the space of polynomials with degree less than or equal to 1.

Thanks to these defintions, we can introduce the following finite element version of the problem \eqref{eq:probreg}:
\begin{equation}\label{eq:probregdis}
\min_{u^h \in W_0^h} J_\gamma^h(u^h):= \frac{1}{p}\int_{\Omega^h} |\nabla u^h|^p\,dx + \int_{\Omega^h}\psi_\gamma (\nabla u^h)\,dx - \int_{\Omega^h}f u^h\,dx.
\end{equation}
\begin{thm}
Problem \eqref{eq:probregdis} has a unique solution $u^h\in W_0^h$.
\end{thm}
\begin{proof}
This result is a direct consequence of the fact that $W_0^h$ is a closed subspace of $W_0^{1,p}(\Omega)$ (see \cite[Sec. 3.2]{glomaro}).\qed
\end{proof}
As stated in the previous section, in finite dimensional spaces is not mandatory to use smoothing steps to construct preconditioned descent algorithms for problems like \eqref{eq:probregdis}. Further, in this case we know that (see \cite{casas,glomaro})
\begin{equation}\label{eq:spdisp>2}
W_0^h\subset W_0^{1,p}(\Omega)\subset H_0^1(\Omega).
\end{equation}
Thanks to this fact, we can consider $W_0^h$ a Hilbert space with the norm induced by $H_0^1(\Omega)$, which we will note by $\|\cdot\|_{W_0^h}$. 

Summarizing, we propose the following algorithm for problem \eqref{eq:probregdis} with $p>2$.
 
\begin{algorithm}\label{algodisp>2}
Initialize $u_0\in W_0^h$ and set $k=0$.  For $k=1,2,\ldots$ do
\begin{enumerate}
\item If $J^{h'}_\gamma (u^h_k)=0$, STOP.
\item Find a search direction $w^h_k$ by solving the following variational equation
\begin{equation}\label{eq:wpdis>2}
\begin{array}{lll}
\int_\Omega (\nabla w^h_k,\nabla v)\,dx =-\langle J^{h'}_\gamma(u^h_k)\,,\, v\rangle_{(W_0^h)^*,W_0^h} \vspace{0.2cm}\\\hspace{0.5cm}= -\int_\Omega |\nabla u^h_k|^{p-2}(\nabla u^h_k,\nabla v)\,dx -g\gamma \int_\Omega \frac{(\nabla u^h_k,\nabla v)}{\max(g,\gamma|\nabla u^h_k|)}\, dx +\int_\Omega f\,v\,dx,\,\,\forall v\in W_0^h.
\end{array}
\end{equation} 
\item Perform an efficient line search technique to obtain $\alpha_k$.
\item Update $u^h_{k+1}:= u^h_k + \alpha_k w^h_k$ and set $k=k+1$.
\end{enumerate}
\end{algorithm}

\begin{prop}\label{eq:wdiscrete}
The equation \eqref{eq:wpdis>2} has a unique solution $w_k^h\in W_0^h$. Furthermore, this solution $w_k^h$ is an admissible descent direction for $J^{h'}_\gamma(u_k)$, \textit{i.e.}, it satisfies that
\[\langle J^{h'}_\gamma(u_k)\,,\, w_k\rangle_{W_0^{h*},W_0^h}<0,\forall k\in \re{N},\]
and the following admissibility condition
\begin{equation}\label{eq:admdis>2}
\frac{\langle J^{h'}_\gamma(u_k)\,,\, w_k\rangle_{W_0^{h*},W_0^h}}{\|w_k\|_{W_0^h}}\underset{k\rightarrow\infty}{\longrightarrow} 0\Rightarrow \|J^{h'}_\gamma(u_k)\|_{W_0^{h*}}\underset{k\rightarrow\infty}{\longrightarrow}  0.
\end{equation}
\end{prop}
\begin{proof}
Existence of a unique solution directly follows from the fact that $W_0^h$ is a Hilbert subspace of $H_0^1(\Omega)$ with the induced norm of this space. Therefore, $w_k^h$ is the Riesz representation of the functional $-J_\gamma^{h'}(u^h_k)$ in the space $W_0^h$ (see \cite[Sec. 3.1]{Huang}). Furthermore, thanks to \eqref{eq:spdisp>2}, from \eqref{eq:wpdis>2} we can conclude that
\begin{equation}\label{eq:desdis1}
\langle J^{h'}_\gamma(u^h_k)\,,\, w_k^h\rangle_{W_0^{h*},W_0^h}= - \|w_k^h\|_{W_0^h}^2< 0, \,\,\forall k\in \re{N},
\end{equation}
which yields that $w_k^h$ is, indeed, a descent direction for $J^{h'}_\gamma(u^h_k)$. Finally, since $w_k^h$ is the Riesz representation of $J^{h'}_\gamma(u^h_k)$ in $W_0^h$, we have that
\[
\|w_k^h\|_{W_0^h} = \|J^{h'}_\gamma(u^h_k)\|_{W_0^{h*}}.
\]
This last identity, together with \eqref{eq:desdis1}, yield that
\[
\langle J^{h'}_\gamma(u^h_k)\,,\, w_k^h\rangle_{W_0^{h*},W_0^h}=- \|J^{h'}_\gamma(u^h_k)\|_{W_0^{h*}}\|w_k^h\|_{W_0^h},
\]
which immediately implies \eqref{eq:admdis>2}.\qed
\end{proof}

\begin{thm}\label{th:convp>2dis}
Let $w^h_k$, $\alpha_k$ and $u^h_k$ generated by Algorithm \ref{algodisp>2}. Then,
\begin{equation}\label{eq:conp>2dis}
\underset{k\rightarrow\infty}{\lim} \, J_\gamma^{h'}(u^h_k)=0.
\end{equation}
\end{thm}
\begin{proof}
Since $w^h_k$ satisfies \eqref{eq:admdis>2} and $\alpha_k$ is calculated by an efficient line search algorithm, admissibility of these two sequences is guaranteed. Therefore, since all the hypothesis of \cite[Th. 2.2]{hinetal} are fulfilled, we can conclude the proof.\qed
\end{proof}

\section{Numerical Implementation}\label{sec:numres}
In this section we discuss all the issues related to the numerical implementation of the algorithms developed in the last section. Further, we present several numerical experiments to show the behavior of these algorithms. Such experiments are concerned with the two cases analyzed during this paper: $1<p<2$ and $p>2$. The case $p=2$ has been widely analyzed, by using a similar regularization approach, in \cite{dlRGpf,dlRG2D,dlRGtd}. Moreover, we focus our experiments on the numerical simulation of the laminar flow of a Herschel-Bulkley fluid in a pipe. Therefore all the experiments have been carried out for a constant function $f$, which represents the linear decay of pressure in the pipe.

\subsection{Discretization issues}
In this section we describe the finite element implementation that we use in all the numerical experiments. Let us start by pointing out that we use the same finite element approach described in Section \ref{sec:algop>2}. Thus, we recall the finite dimension space 
\[
W_0^h:=\{v\in C(\overline{\Omega^h})\,:\, \mbox{ $v|_{\tau}\in \re{P}_1, \forall \tau \in T^h$ and $v=0$ on $\partial \Omega^h$}  \},
\]
where $\re{P}_1$ is the space of polynomials with degree less than or equal to 1. We note the basis functions of $W_0^h$ by $\varphi_j$, $j=1,\ldots,n$ and we assume that $card(T^h)=m$. Further, we use the notation $\overrightarrow{u}$ for the coefficients of the approximated functions $u^h$.

By following ideas in \cite[Sec. 4]{dlRGpf}, we use the following discrete version of the gradient 
\begin{equation}\label{eq:gradiscrete}
\nabla^{h}:=\begin{pmatrix}
 \partial_{1}^{h}\\\partial_{2}^{h}
\end{pmatrix} \in\re{R}^{2m\times n},
\end{equation}
where $\partial_{1}^{h}:=\left. \frac{\partial\varphi_i(x)}{\partial x_1} \right|_{\tau_k}$ and
$\partial_{2}^{h}:=\left. \frac{\partial\varphi_i(x)}{\partial x_2} \right|_{\tau_k}$, for
$i=1,\ldots,n$ and $\tau_k\in T^h$. Note that $\left. \frac{\partial\varphi_i(x)}{\partial x_1}
\right|_{\tau_k}$ and $\left. \frac{\partial\varphi_i(x)}{\partial x_2}\right|_{\tau_k}$ are the
constant values of $\frac{\partial\varphi_i(x)}{\partial x_1}$ and $\frac{\partial\varphi_i(x)}{\partial x_1}$ in
each triangle $\tau_k$, respectively. Consequently, $\nabla^h \overrightarrow{u}$ is the approximation of $\nabla u^h(x)$. 

Next, let us introduce the function $\xi:\re{R}^{2m}\rightarrow \re{R}^m$ given by
\[\xi(w)_k = |(w_k,w_{k+m})|^\top,\,\, k=1,\ldots,m.\]
Therefore, we calculate $|\nabla u^h(x)|$ by $\xi(\nabla^h \overrightarrow{u})$. Note that $\xi(\nabla^h\overrightarrow{u})_k$ represents the value of $|\nabla u^h(x)|$ at each triangle $\tau_k\in T^h$.

Finally, we discuss the implementation of $\int_\Omega (\epsilon+ |\nabla u|)^{p-2} (\nabla w, \nabla v)\,dx$, $\int_\Omega |\nabla u|^{p-2} (\nabla u, \nabla v)\,dx$ and $g \gamma \int_\Omega  \frac{(\nabla u,\nabla v)}{\max(g,\gamma |\nabla u|)}\, dx$. By using the Galerkin's method, we obtain the following
\begin{itemize}
\item $
\int_\Omega (\epsilon+ |\nabla u|)^{p-2} (\nabla w,\nabla \varphi_j)\,dx \approx \sum_{i=1}^{n} w_i \sum_{\tau_k\in T^h}\int_{\tau_k} (\epsilon+\xi(\nabla^h\overrightarrow{u})_k)^{p-2} (\nabla \varphi_i,\nabla \varphi_j)\,dx,\vspace{0.2cm}
$
\item $
\int_\Omega |\nabla u|^{p-2} (\nabla u,\nabla \varphi_j)\,dx \approx \sum_{i=1}^{n} u_i \sum_{\tau_k\in T^h}\int_{\tau_k} (\xi(\nabla^h\overrightarrow{u})_k)^{p-2} (\nabla \varphi_i,\nabla \varphi_j)\,dx
$
and\vspace{0.2cm}
\item $
\int_\Omega g \gamma\frac{(\nabla u,\nabla \varphi_j)}{\max(g,\gamma|\nabla u|)}\,dx \approx  \sum_{i=1}^{n} u_i \sum_{\tau_k\in T^h}g\gamma\int_{\tau_k} \frac{(\nabla \varphi_i,\nabla \varphi_j)}{\max(g,\gamma\xi(\nabla^h\overrightarrow{u})_k)} \,dx,
$
\end{itemize}
for $j=1,\ldots,n.$ Next, note that the terms $ (\epsilon+\xi(\nabla^h\overrightarrow{u})_k)^{p-2}$, $ (\xi(\nabla^h\overrightarrow{u})_k)^{p-2}$ and $\max(g,\gamma\xi(\nabla^h\overrightarrow{u})_k)$ are constant at every triangle $\tau_k$. 

Thus, by using ideas in \cite{cc50}, we obtain a matrix approximation $A^h_{\epsilon,u}\in \re{R}^{n\times n}$, $A^h_u\in \re{R}^{n\times n}$ and $A_{u,\max}^h\in \re{R}^{n\times n}$, for any of the forms in the expression above. The entries of these matrices are given by
\begin{itemize}
\item $(a_{\epsilon,u})_{i,j}=\sum_{\tau_k\in T^h} (\epsilon+ \xi(\nabla^h\overrightarrow{u})_k)^{p-2}\int_{\tau_k}  (\nabla \varphi_i,\nabla \varphi_j)\,dx$,
\item $(a_u)_{i,j}=g\gamma\sum_{\tau_k\in T^h} (\xi(\nabla^h\overrightarrow{u})_k)^{p-2}\int_{\tau_k}  (\nabla \varphi_i,\nabla \varphi_j)\,dx$ and
\item $(a_{u,\max})_{i,j}=g\gamma \sum_{\tau_k\in T^h}\frac{1}{\max(g,\gamma\xi(\nabla^h\overrightarrow{u})_k)}\int_{\tau_k} (\nabla \varphi_i,\nabla \varphi_j) \,dx$.
\end{itemize}
Finally, by following ideas in \cite{barret}, we approximate the right hand side as follows
\[
\int_\Omega f^h\varphi_j\,dx\approx \sum_{\tau\in T^h} Q_\tau (f\varphi_j),\,\, j=1,\ldots, n,
\]
where the quadrature rule $Q_\tau$ is given by
\[
Q_\tau(v) = \frac{1}{3} meas(\tau) \sum_{i=1}^3 v(a_i),\,\,\mbox{with $a_i,\,i =1,\ldots, 3$ the vertices of $\tau\in T^h$}.
\]
\begin{rem}
It is remarkable that due to the proposed structure, the Algorithms \ref{algo1<p<2} and \ref{algodisp>2} only need to solve one linear system at each iteration.  In fact, Algorithm \ref{algo1<p<2} and  Algorithm \ref{algodisp>2} require the solution of linear systems like
\[
A^h_{\epsilon,u}\,\overrightarrow{w}= \eta_1^h, \mbox{   and   } A^h\,\overrightarrow{w} = \eta_2^h,
\] 
respectively. Here, $A^h_{\epsilon,u}$ is given above, $A^h$ is the classical stiffness matrix and $\eta_1^h$ and $\eta_2^h$ are the F.E.M. approximation of the right hand side of equations \eqref{eq:wp<2algo} and \eqref{eq:wpdis>2}, respectively.  Note that matrix $A^h_{\epsilon,u}$ depends on $u_k$,  but does not depend on $w_k$. Further, $A^h$ does not depend neither on $u_k$ nor in $w_k$. This fact implies that the linear systems can be easily solved by any direct or iterative method and does not represent a large computational effort. 
\end{rem}

\begin{rem}(Stopping Criterion)\label{rem:stop}
We stop the Algorithms \ref{algo1<p<2} and \ref{algodisp>2} as soon as the expression $\frac{|J'_{\gamma,h}(\overrightarrow{u}_k)|}{|J'_{\gamma,h}(\overrightarrow{u}_0)|}$ is reduced by a factor of $10^{-6}$.  Here $J'_{\gamma,h}(\overrightarrow{u}_k)$ stands for the FEM discrete version of $J'(u_k)$ and is given by
\[
J'_{\gamma,h}(\overrightarrow{u}_k):= A_{u}^h\overrightarrow{u}_k + A_{u,\max}^h \overrightarrow{u}_k - \overrightarrow{f}. 
\]
This kind of stopping criterion is popular for steepest descent algorithms, since it is easy to implement, and it provides enough information about the convergence behavior of the algorithm (see \cite{kelley}).
\end{rem}

\subsection{Line search algorithms}
As stated in Section \ref{sec:preconalgo}, we need to focus on the implementation of efficient inexact line search methods. One typical technique is the backtracking line-search algorithm.  The general idea behind this approach is to take $\alpha_k=1$. Then, if $u_k+\alpha_k w_k$ is not acceptable, in the sense that a descent condition on $J_\gamma$ is not fulfilled, $\alpha_k$ is reduced (``backtracked'') until $u_k+\alpha_k w_k$ is acceptable.

We propose to use an algorithm which uses polynomial models of the objective functional for backtracking, which is detailed in \cite[Sec. 6.3.2]{densch}. In this section, we briefly describe this algorithm and the main ideas behind it.

The central discussion in a backtracking algorithm is how to reduce $\alpha_k$. Usually, the backtracking algorithm is implemented by taking $\alpha_{k}=\frac{1}{2^k}\alpha_k$, so $\alpha_k$ is reduced to half at each iteration. This procedure can be inefficient since usually needs several iterations to achieve convergence, and, moreover, the step sizes can be very small.

In this paper, following ideas in \cite[Sec. 6.3.2]{densch}, we propose a reduction strategy for $\alpha_k$ based on polynomial models of the objective function. 

Let us start by introducing the following function
\[
\varphi_k(\alpha):= J_\gamma(u_k + \alpha w_k).
\]
Next, by using the current information of $J_\gamma$, we take $\alpha_k$ as the approximation to the value that minimizes $\varphi_k(\alpha)$, \textit{i.e.}, $\alpha_k\approx \argmin \varphi_k(\alpha)$. 

First, note that the following information about $\varphi_k$ is available.
\begin{equation}\label{eq:alarmijo1}
\varphi_k(0)= J(u_k) \mbox{   and   }  \varphi_k'(0)= \langle J'_\gamma(u_k)\,,\, w_k\rangle.
\end{equation}
Further, once we calculate $J_\gamma(u_k + w_k)$, we know that
\begin{equation}\label{eq:alarmijo2}
\varphi_k(1)=J_\gamma(u_k + w_k).
\end{equation}
Next, if $J_\gamma(u_k + w_k)$ does not satisfy the descent condition (\textit{i.e.}, $\varphi_k(1)>\varphi_k(0) + \sigma_1 \varphi_k'(0)$), we construct the following quadratic model for $\varphi_k$ by using \eqref{eq:alarmijo1} and \eqref{eq:alarmijo2}. 
\[
m_2(\alpha):= (\varphi_k(1)- \varphi_k(0) - \varphi_k'(0))\alpha^2 + \varphi_k'(0)\alpha + \varphi_k(0).
\]
It is easy to prove that
\begin{equation}\label{eq:alarmijo3}
\tilde{\alpha}_2=\frac{-\varphi_k'(0)}{2(\varphi_k(1)-\varphi_k(0)- \varphi_k'(0))}
\end{equation}
is a stationary point of $m_2$, \textit{i.e.,} it satisfies that $m_2'(\tilde{\alpha})=0$. Moreover, we have that
\[
m_2''(\alpha)=2(\varphi_k(1)-\varphi_k(0)-\varphi_k'(0))>0,
\]
since $\varphi_k(1)>\varphi_k(0)+\sigma_1 \varphi_k'(0)>\varphi_k(0)+\varphi_k'(0)$. Thus, we conclude that $\tilde{\alpha}_2$ minimizes the model $m_2$ and, since $\varphi_k'(0)<0$, we have that $\tilde{\alpha}_2>0$. Consequently, we take $\alpha_k:=\tilde{\alpha}_2$. 

Now, since $\varphi_k(1)> \varphi_k(0) + \alpha \varphi_k'(0)$, from \eqref{eq:alarmijo3}, we have that
\[
\tilde{\alpha}_2< \frac{1}{2(1-\sigma_1)}.
\]
This fact implies, provided $\varphi_k(1)\geq \varphi_k(0)$, that $\tilde{\alpha}\leq \frac{1}{2}$. Therefore, \eqref{eq:alarmijo3} gives an implicit upper bound of $\approx\frac{1}{2}$ for $\tilde{\alpha}_2$ on the first backtrack. On the other hand, if $\varphi_k(1)>>\varphi_k(0)$, $\tilde{\alpha}$ can be very small. This fact suggests that $\varphi_k(\alpha)$ is probably poorly modeled by a quadratic function is this region. In order to avoid too small steps, we impose a lower bound of $\frac{1}{10}$. Therefore, if at the first backtrack at each iteration we have that $\tilde{\alpha}_2\leq 0.1$, the algorithm next tries $\alpha_k=\frac{1}{10}$ 

Now, suppose that $\varphi(\tilde{\alpha}_2)$ does not satisfy \eqref{eq:Wolfe1}, which implies that we need to backtrack again. In this case, we have the following information available: $\varphi_k(0)= J(u_k)$, $\varphi_k'(0)= \langle J'_\gamma(u_k)\,,\, w_k\rangle$ and the last two values of $\varphi(\alpha)$. Therefore, we use a cubic model of $\varphi$ fitting all these pieces of information, and set $\alpha_k$ to be the minimizer of this new model. This procedure is justified since a cubic polynomial can perform  better when modelling situations where $J_\gamma$ has negative curvature, which are likely when \eqref{eq:Wolfe1} is not achieved for two possible values of $\alpha$ (see \cite[pp. 128]{densch}).

The construction of this cubic model is as follows. Let $\alpha_p$ and $\alpha_{2p}$ be the last two previous values of $\alpha_k$. Then, the cubic that fits  $\varphi_k(0)$,  $\varphi'_k(0)$,  $\varphi_k(\alpha_p)$ and $\varphi_k(\alpha_{2p})$ is given by
\[
m_3(\alpha):= c\alpha^3 + d\alpha^2 + \varphi_k'(0)\alpha + \varphi_k(0),
\]
where
\[
\begin{pmatrix}
c\\d
\end{pmatrix} = \frac{1}{\alpha_p - \alpha_{2p}} \begin{pmatrix}
\frac{1}{\alpha_p^2} & \frac{-1}{\alpha_{2p}^2}\\\frac{-\alpha_{2p}}{\alpha_p^2} & \frac{\alpha_p}{\alpha_{2p}^2}
\end{pmatrix}\begin{pmatrix}
\varphi_k(\alpha_p)-\varphi_k(0)-\varphi'_k(0)\alpha_p \\\varphi_k(\alpha_{2p})-\varphi_k(0)-\varphi'_k(0)\alpha_{2p}.  
\end{pmatrix}
\]
Further, it is easy to prove that the minimizer of $m_3$ is 
\begin{equation}\label{eq:alarmijo4}
\tilde{\alpha}_3= \frac{-d+\sqrt{d^2 - 3c\varphi_k'(0)}}{3c}.
\end{equation}
In \cite{densch} is established that if $\varphi(\alpha_p)\geq \varphi(0)$, then $\tilde{\alpha}_3< \frac{2}{3}\alpha_p$, but this reduction is considered too small. Therefore, we impose the upper bound $b=0.5$, which implies that if $\tilde{\alpha}_3> \frac{1}{2}\alpha_p$, we set $\alpha_k = \frac{1}{2}\alpha_p$. Also, since $\tilde{\alpha}_3$ can be an arbitrarily small fraction of $\alpha_p$, we again impose the lower bound $a=\frac{1}{10}$, \textit{i.e.}, if $\tilde{\alpha}_3< \frac{1}{10}\alpha_p$, we set $\alpha_k = \frac{1}{10}\alpha_p$. 

Summarizing, we have the following line search algorithm.

\begin{algorithm}\label{algo:armijo}
Let $\sigma_1\in (0,\frac{1}{2})$ and set $\alpha_0=1$. 
\begin{enumerate}
\item Decide wheter $J_\gamma(u_k + \alpha_k)> J_\gamma (u_k) + \sigma_1 \alpha_k \langle J'_\gamma(u_k)\,,\, w_k\rangle$ holds. If so, STOP and set $\alpha_k=\alpha_0$. If not:
\item Decide wheter steplength is too small. If so, STOP and terminate algorithm: routine failed to locate satisfactory $x_{k+1}$ sufficiently distinct from $x_k$. If not:
\item Decrease $\alpha$ by a factor between 0.1 and 0.5 as follows:
\begin{enumerate}
\item On the first backtrack: set $\alpha_k:=\tilde{\alpha}_2= \argmin m_2(\alpha)$, but constrain the new $\alpha_k$ to be $\geq 0.1$.
\item On all the subsequent backtracks: set $\alpha_k:= \tilde{\alpha}_3 =\argmin m_3(\alpha)$, but constraint the new $\alpha_k$ to be in $[0.1\alpha_p\,,\, 0.5 \alpha_p]$.
\end{enumerate}
\item Return to step 1.
\end{enumerate}
\end{algorithm}
Here, the parameter $\sigma_1$ is set quite small, usually in the order of $10^{-4}$. Further, \eqref{eq:alarmijo4} is never imaginary if $\sigma_1$ is less than $\frac{1}{4}$ (see \cite[pp. 129]{densch})

Note that this algorithm only implements the first Wolfe-Powell condition \eqref{eq:Wolfe1}. The curvature condition \eqref{eq:Wolfe2} is not usually implemented because the backtracking technique avoids excessively small steps. It is established that the bounds in the algorithm on the amount of each calculation of $\alpha$ make the curvature condition to hold  (for further details and examples see \cite[pp. 126-129]{densch} and the references therein).

\subsection{Numerical Results: Case $1<p<2$}
In this section, we focus on the behavior of Algorithm \ref{algo1<p<2}. In the next experiments, we consider that the problem \eqref{eq:probreg} represents the flow of a Herschel-Bulkley fluid with $1<p<2$, so we are in the case of a shear-thinning material. Further, we consider a constant $f$, which represents the linear decay of pressure in the pipe. In this context, the constant $g$ plays the role of the plasticity threshold and it is modelled by the Oldroy number (see \cite{Huilgol}). For further details in the mechanics of these problems, we refer the reader to \cite{Chhabra,dlRGpf,Huilgol} and the references therein. 

Hereafter, we use uniform triangulations described by $h$, the radius of the inscribed circumferences of the triangles in the mesh. In the next examples, we use the values
$\gamma=10^{3}$ and $\epsilon=10^{-6}$, and we initialize the algorithm \ref{algo1<p<2} with the solution of the Poisson problem
$-\Delta u_0^h=f^h$.  Further, we stop the algorithm by using the stopping criteria described in Remark \ref{rem:stop}. 

\subsubsection{Experiment 1}
In this experiment, we set $\Omega\subset \re{R}^2$ to be the unit ball, and we compute the flow of a Herschel-Bulkley material with $p=1.75$. We analyze the behavior of the algorithm with $g=0.2$ and $f=1$, and we use a mesh given by $h\approx 0.0086$.
\begin{figure}
\begin{center}
\includegraphics[width=80mm, height=60mm]{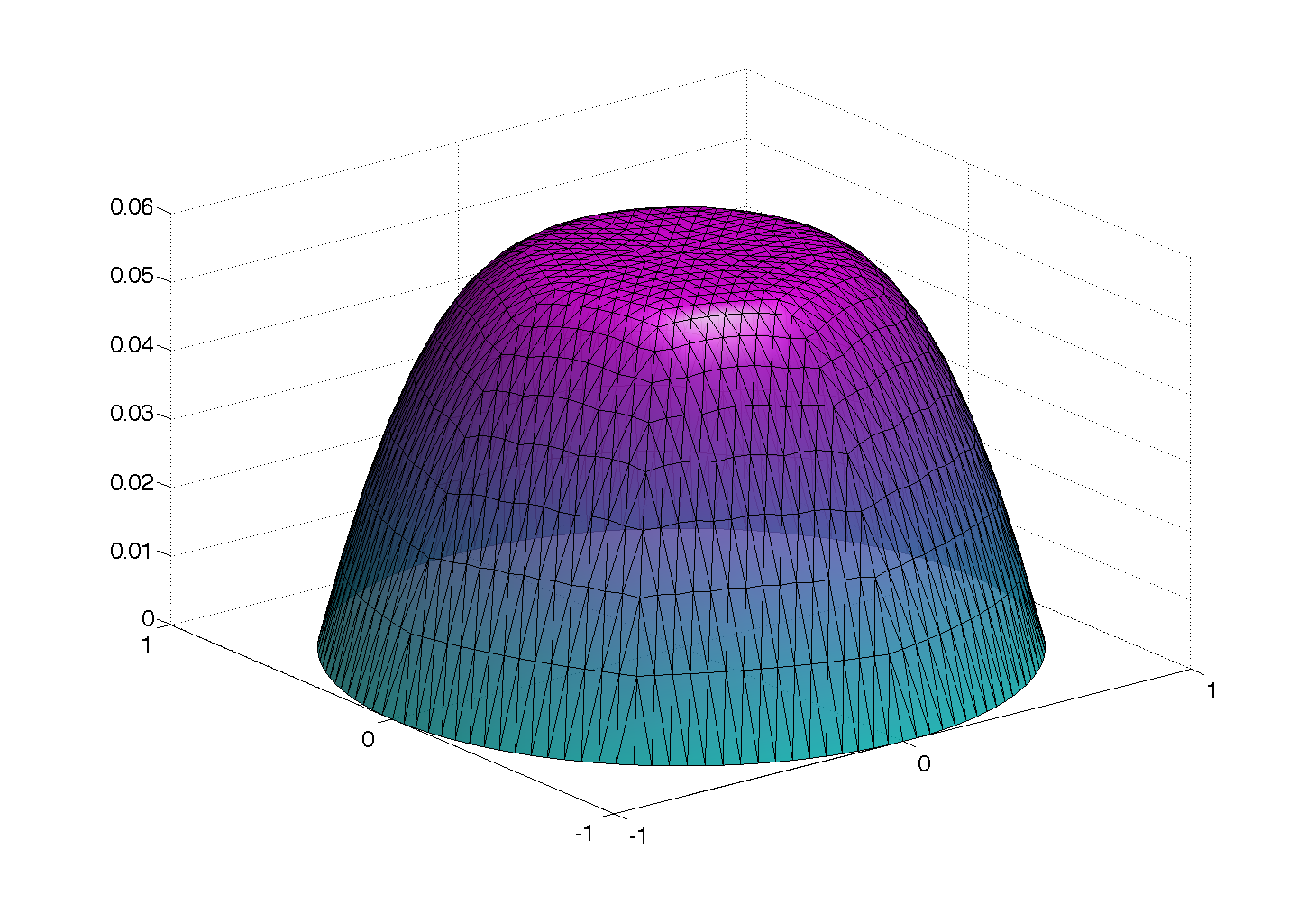}\hspace{0.cm}\includegraphics[width=75mm, height=55mm]{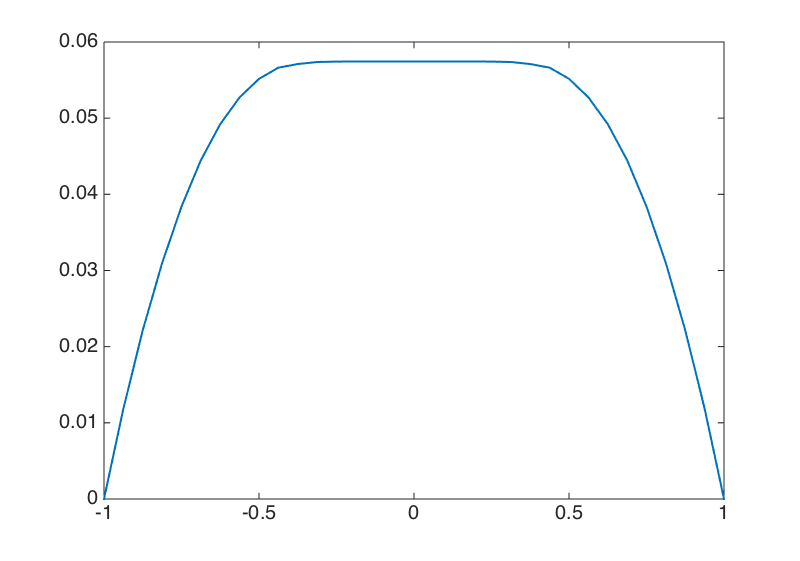}
\end{center}
\caption{Calculated vleocity $u$  (left) and velocity profile along the diameter of the pipe (right). Parameters: $p=1.75$, $g=0.2$, $\gamma=10^3$ and $\epsilon=10^{-6}$.}\label{fig:up175}
\end{figure}

\begin{table}
\begin{center}
\begin{tabular}{|c|c|c|c|c|c|}
\hline
it. & $\frac{|J'_{\gamma,h}(\overrightarrow{u}_k)|}{|J'_{\gamma,h}(\overrightarrow{u}_0)|}$ & $J_{\gamma,h}(\overrightarrow{u}_k)$ & $\alpha_k$ & l.s. it.\\
\hline
 1 &  2.070e-3  & -0.022393 & 1.0000 & 0\\ 
 2 &  8.758e-3  & -0.027232 & 0.4199 & 1\\
 3 &  2.959e-3  & -0.028522 & 0.3390 & 1\\
 4 &  5.582e-4  & -0.028778 & 0.2600 & 1\\
 5 &  3.532e-4  & -0.028911 & 0.1231 & 2\\
 6 &  7.549e-4  & -0.029028 & 0.1864 & 1\\
 7 &  6.590e-4  & -0.029057 & 0.0788 & 2\\
 8 &  4.865e-4  & -0.029091 & 0.0558 & 2\\
 9 &  1.179e-4 & -0.029101 & 0.0485 & 2\\
 10 & 6.655e-7 & -0.029107 & 0.0424 & 2\\
  \hline
\end{tabular}
\caption{Convergence behavior for Algorithm \ref{algo1<p<2}. Parameters: $p=1.75$, $g=0.2$, $\gamma=10^3$ and $\epsilon=10^{-6}$.}\label{tab:conv1<p<2ex1}
\end{center}
\end{table}

The resulting velocity function and the velocity profile along the diameter of the pipe are displayed in Figure \ref{fig:up175}. The graphics illustrate the expected mechanical properties of the material, i.e., since the shear stress transmitted by a fluid layer decreases toward the center of the pipe, the Herschel-Bulkley fluid moves like a solid in that sector. This effect explains the flattening of the velocity in the center of the pipe.

\begin{figure}
\begin{center}
\includegraphics[width=70mm, height=50mm]{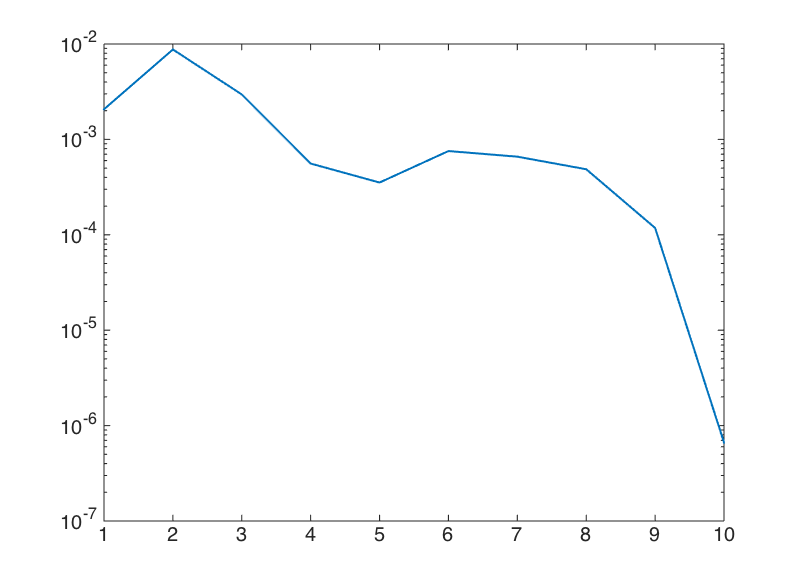}
\end{center}
\caption{Calculated residual $|J'_{\gamma,h}({u}_k)|/|J_{\gamma,h}'(u_0)|$ for: Algorithm \ref{algo1<p<2} (left) and Wolfe-Powell (right). Parameters: $p=1.75$, $g=0.2$, $\gamma=10^3$ and $\epsilon=10^{-6}$.}\label{fig:errp175}
\end{figure}

In Table \ref{tab:conv1<p<2ex1}, we show the number of iterations that Algorithm \ref{algo1<p<2} needs to achieve convergence. We also show the value of $\frac{|J'_{\gamma,h}(\overrightarrow{u}_k)|}{|J'_{\gamma,h}(\overrightarrow{u}_0)|}$, the value of $J_{\gamma,h}(\overrightarrow{u}_k)$, the value of the step $\alpha_k$ and the number of inner iterations needed by Algorithm \ref{algo:armijo}. As expected, the value of the functional is monotonically reduced at every iteration. The residual behaves typically as in a steepest descent algorithm, as shown in Figure \ref{fig:errp175}. However, $|J'_{\gamma,h}(\overrightarrow{u}_k)|/|J'_{\gamma,h}(\overrightarrow{u}_0)|$ decays faster in the last iterations. This fact suggests, at least experimentally, that this algorithm has a fast local convergence rate. The step $\alpha_k$ is also monotonically decreasing and the line search Algorithm \ref{algo:armijo} needs no more than two inner iterations to calculate the step. 

\begin{table}
\begin{center}
\begin{tabular}{|c|c|c|c|c|}
\hline
$\epsilon$& it.  & $J(u)$& $|J'_{\gamma,h}(u_k)|/|J_{\gamma,h}'(u_0)|$\\
\hline
  1e-4  & 10 &  -0.029107 & 1.114e-6 \\
  1e-5  & 10 &  -0.029107 & 7.064e-7 \\ 
  1e-6  & 10 &  -0.029107 & 6.655e-7 \\
\hline
\end{tabular}
\caption{Dependence on $\epsilon$ for Algorithm \ref{algo1<p<2}. Parameters: $p=1.75$, $g=0.2$ and $\gamma=10^3$}\label{tab:epsilon}
\end{center}
\end{table}

Finally, in Table \ref{tab:epsilon} we compare the behavior of the Algorithm \ref{algo1<p<2} for different values of the parameter $\epsilon$. It is clear that the performance of the Algorithm is similar in the three cases shown. Some small improvement can be seen, though, for small values of $\epsilon$. 

Let us emphasize that our method requires a low computational effort to produce results which are in good agreement with previous contributions (\textit{e.g.},\cite{Huilgol}). In fact, we only need to solve one linear system per iteration and the line search strategy needs two iterations in average. 

\subsubsection{Experiment 2}
In this experiment, we set $\Omega$ to be the unit square $(0,1)\times (0,1)$, and we compute the flow of a Herschel-Bulkley material given by $p=1.5$. We fix $f=3$, and we focus on the behaviour of the algorithm in different meshes, since we are interested in showing, at least numerically, the mesh independence of our algorithm. It is known that smaller values of $p$ imply that the functional loses regularity, making the problem a bit more challenging.  In fact, as $g$ grows, the contribution of the less regular component of the functional $\int_\Omega \psi_\gamma(\nabla u)\,dx$ increases. This fact complicates the numerical approximation of the problem. Therefore, we test our algorithm with several values of $g$ to show the versatility of our approach.

\begin{table}
\begin{center}
\begin{tabular}{|c|c|c|c|}
\hline
$g=0.1$ & $h_1$ & $h_2$ & $h_3$ \\
\hline
Iter. num. & 9 & 9 & 9  \\
$|J'_{\gamma,h}(u_k)|/|J_{\gamma,h}'(u_0)|$& 1.147e-6 & 1.667e-6  & 1.661e-6 \\
$J_{\gamma,h}(u_k)$ & -0.0395 & -0.0414 &  -0.0416  \\
\hline
$g=0.2$ & $h_1$ & $h_2$ & $h_3$ \\
\hline
Iter. num. & 9 & 8  & 8  \\
$|J'_{\gamma,h}(u_k)|/|J_{\gamma,h}'(u_0)|$&  1.490e-6 & 7.059e-7 & 3.976e-6 \\
$J_{\gamma,h}(u_k)$ & -0.0217 & -0.0231 & -0.0233 \\
\hline
$g=0.3$ & $h_1$ & $h_2$ & $h_3$ \\
\hline
Iter. num. & 18 & 19 & 19 \\
$|J'_{\gamma,h}(u_k)|/|J_{\gamma,h}'(u_0)|$& 4.232e-6 & 4.393e-6 & 1.342e-6\\
$J_{\gamma,h}(u_k)$ & -0.0105 & -0.0115 & -0.0116 \\
\hline
\end{tabular}
\caption{Convergence behavior for Algorithm \ref{algo1<p<2}. Parameters: $p=1.5$ and $\gamma=10^3$}\label{tab:ex1p<2mesh}
\end{center}
\end{table}

In Table \ref{tab:ex1p<2mesh}, we present the main features of Algorithm \ref{algo1<p<2} for several values of $g$ and different mesh sizes: $h_1\approx 0.0133$, $h_2\approx 0.0047$ and $h_3\approx 0.0029$. As expected, the number of iterations that the Algorithm needs to achieve convergence increases as $g$ does. However, for a given $g$, the number of iterations is very stable as the mesh size decreases. Also, the evolution of $|J'_{\gamma,h}(\overrightarrow{u}_k)|/|J'_{\gamma,h}(\overrightarrow{u}_0)|$ is quite similar at every mesh, as shown in Figure \ref{fig:errp15}. These facts show the robustness of our approach and numerically verify the mesh independence of the algorithm. 

\begin{figure}
\begin{center}
\includegraphics[width=50mm, height=40mm]{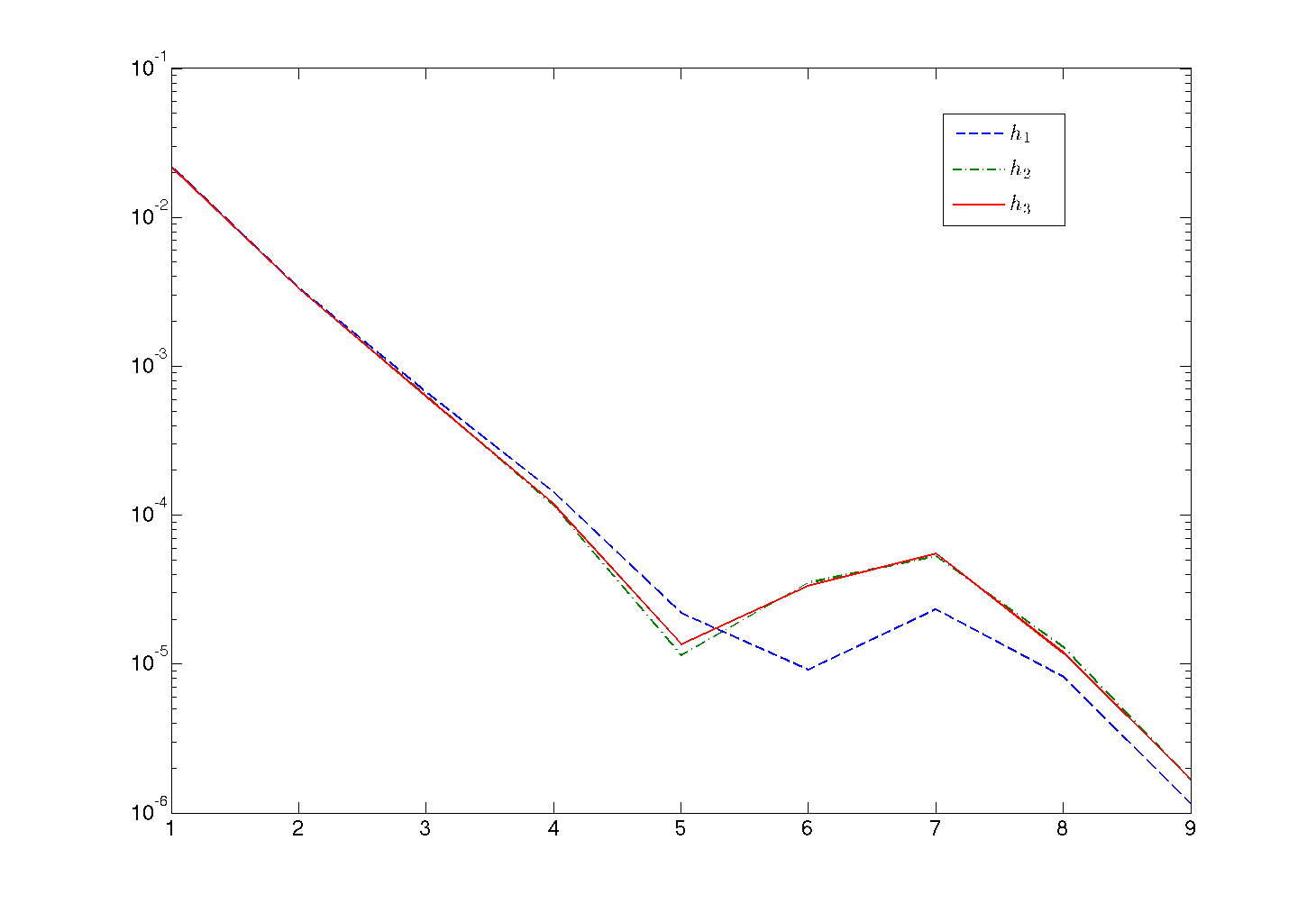}\hspace{0.1cm}\includegraphics[width=50mm, height=40mm]{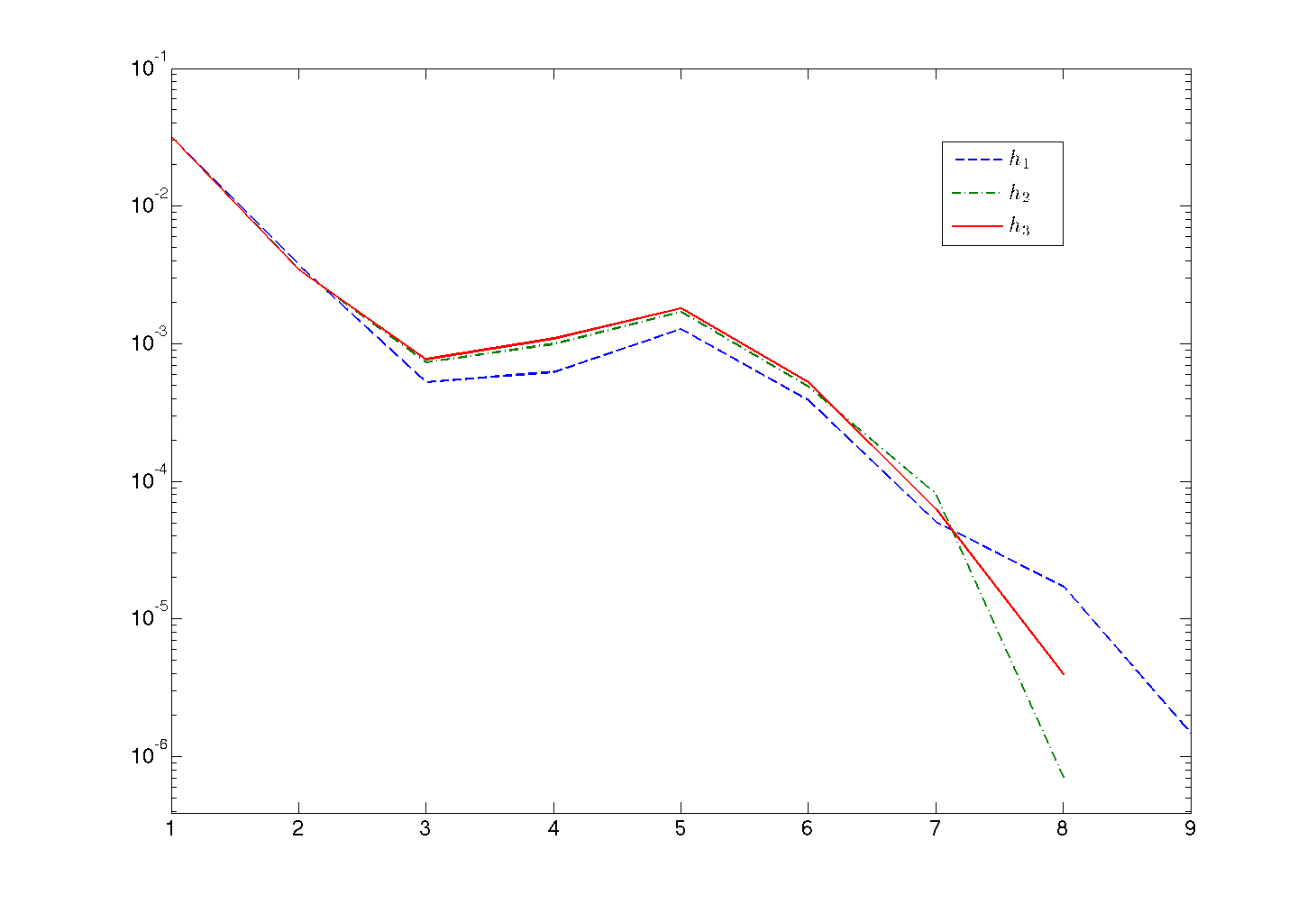}\hspace{0.1cm}
\includegraphics[width=50mm, height=40mm]{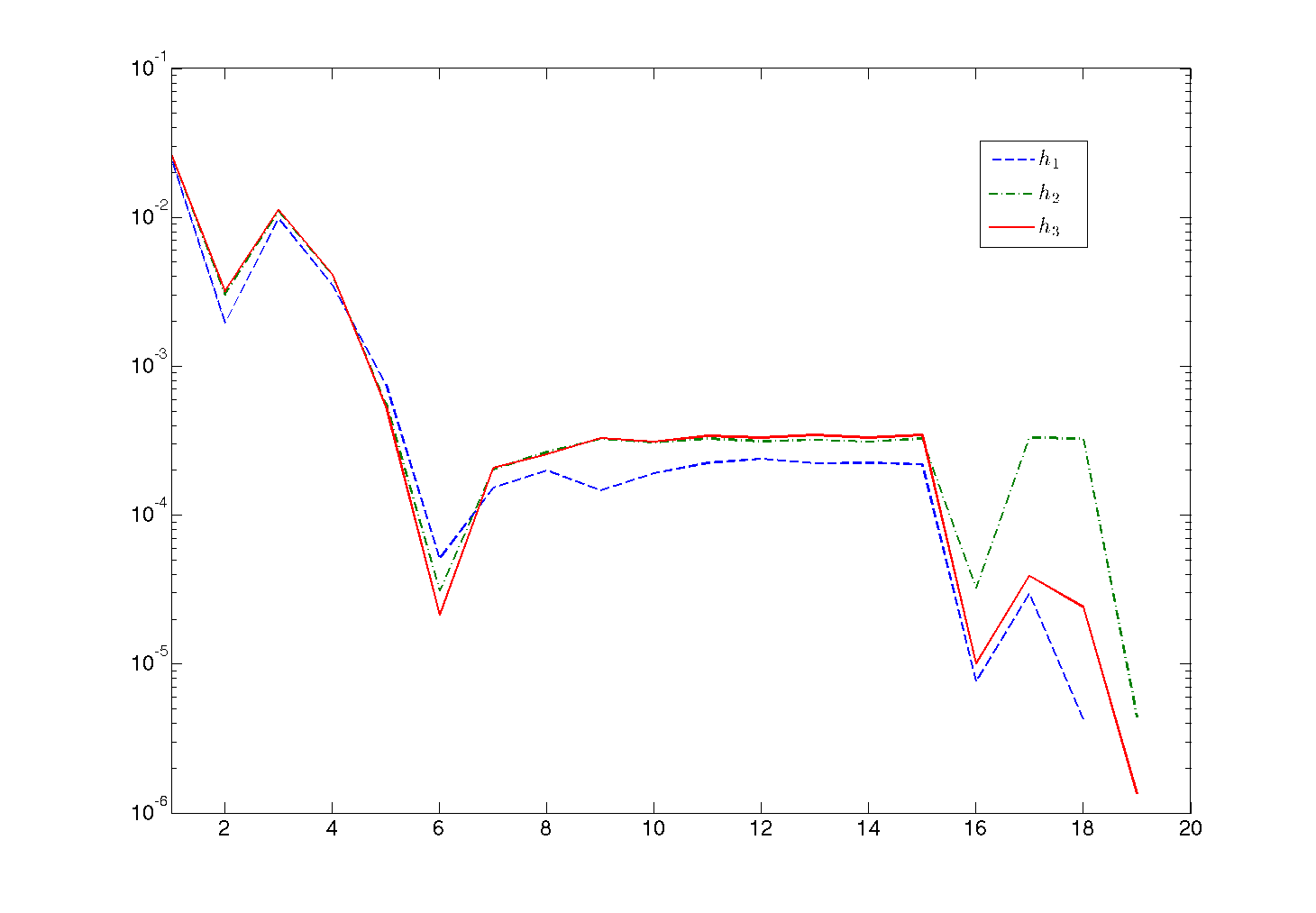}
\end{center}
\caption{Calculated residuals $\|J'(u_k)\|$, in different meshes, for $p=1.5$ and $g=0.1$ (left), $g=0.2$ (center) and $g=0.3$ (right). Parameters: $\gamma=10^3$.}\label{fig:errp15}
\end{figure}

\begin{figure}
\begin{center}
\includegraphics[width=60mm, height=50mm]{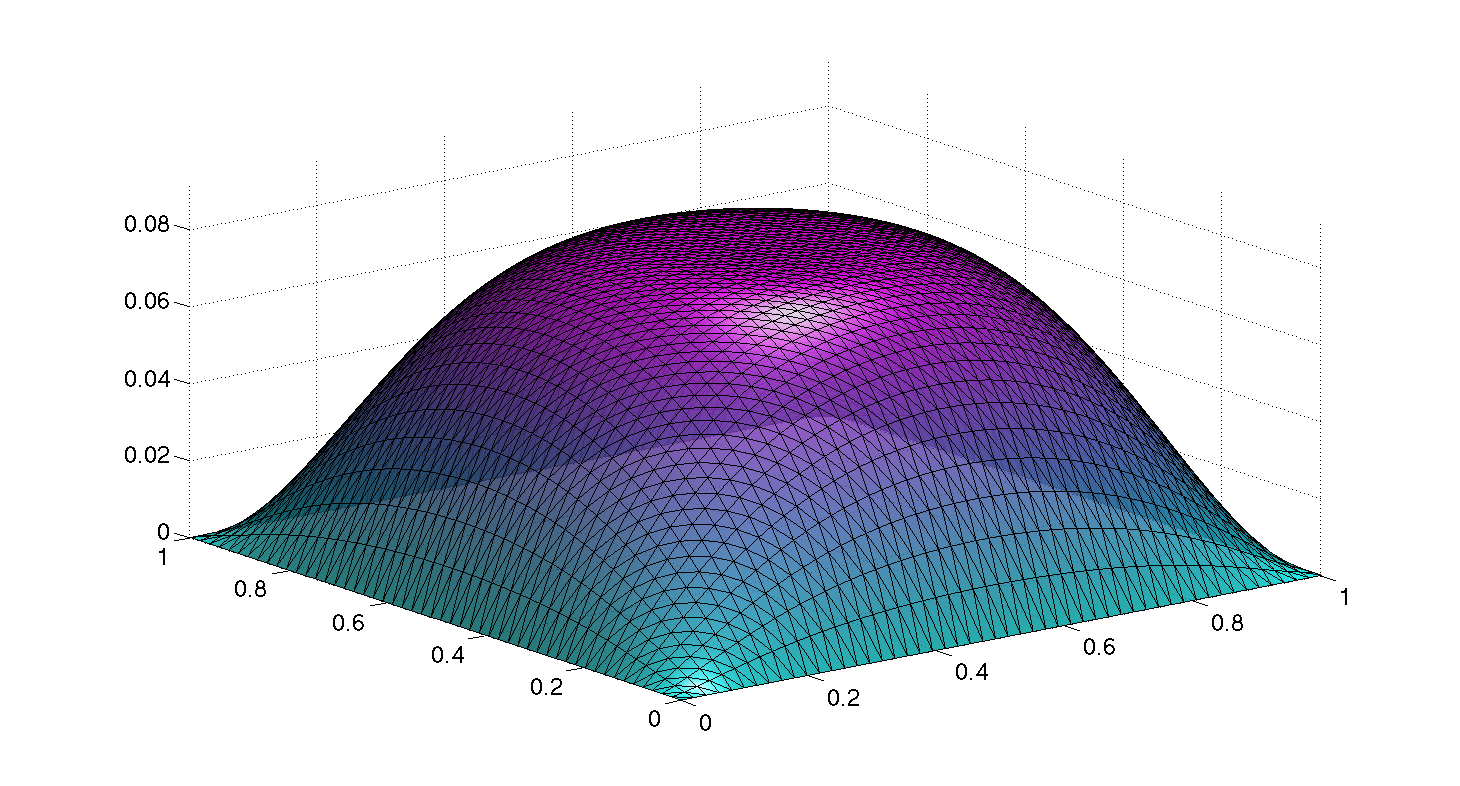}\hspace{0.5cm}\includegraphics[width=60mm, height=50mm]{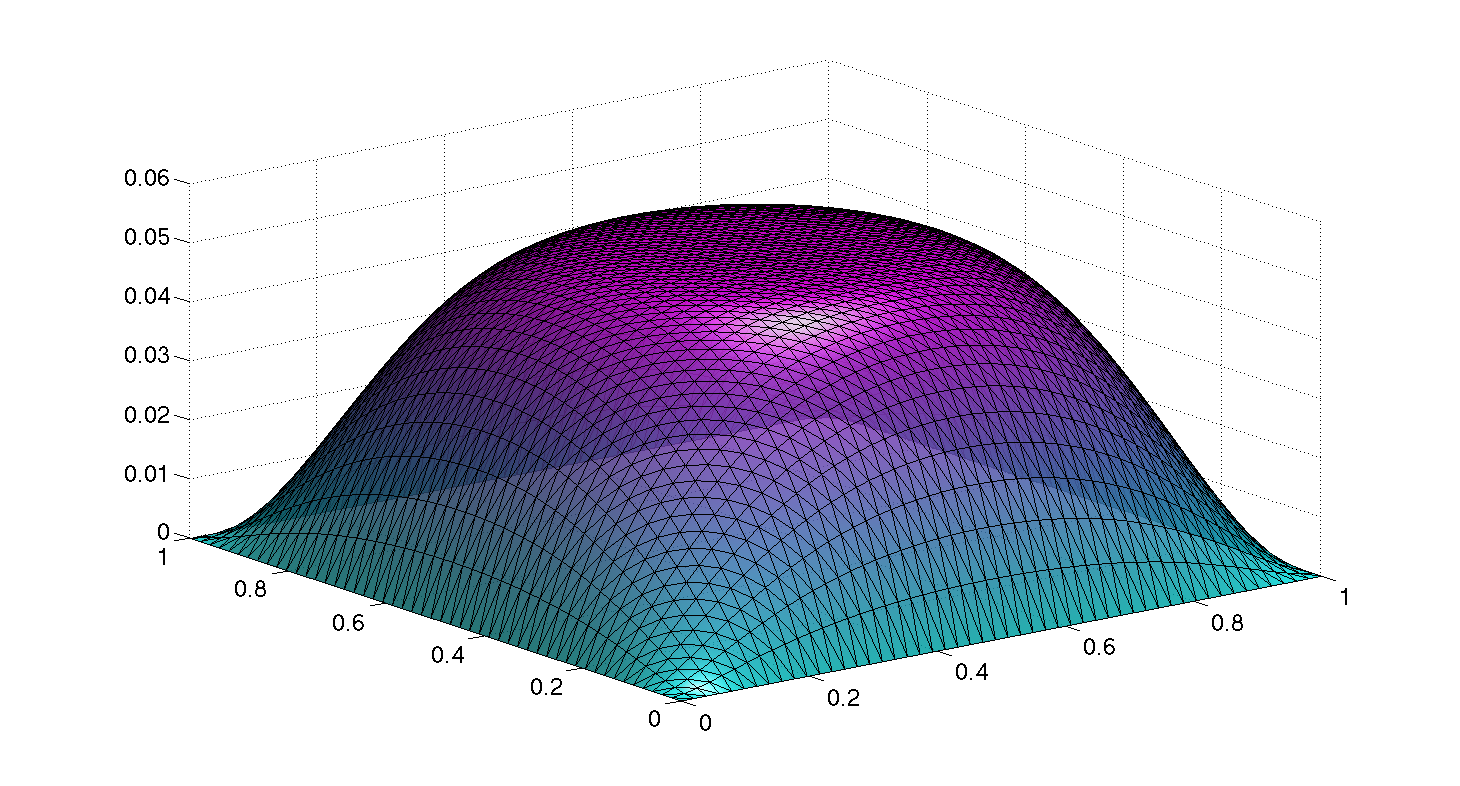}\hspace{0.5cm}
\includegraphics[width=60mm, height=50mm]{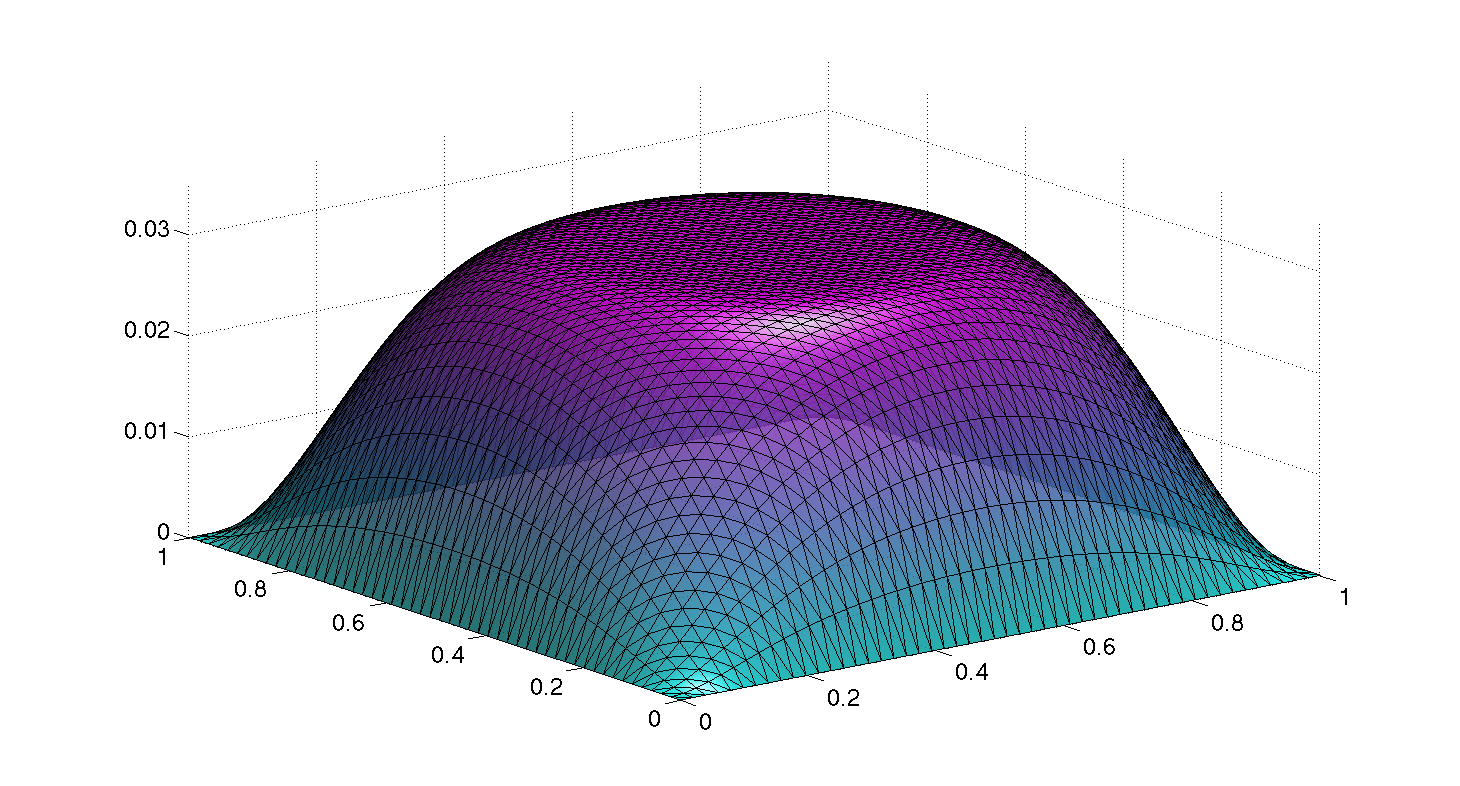}\hspace{0.5cm}
\includegraphics[width=55mm, height=45mm]{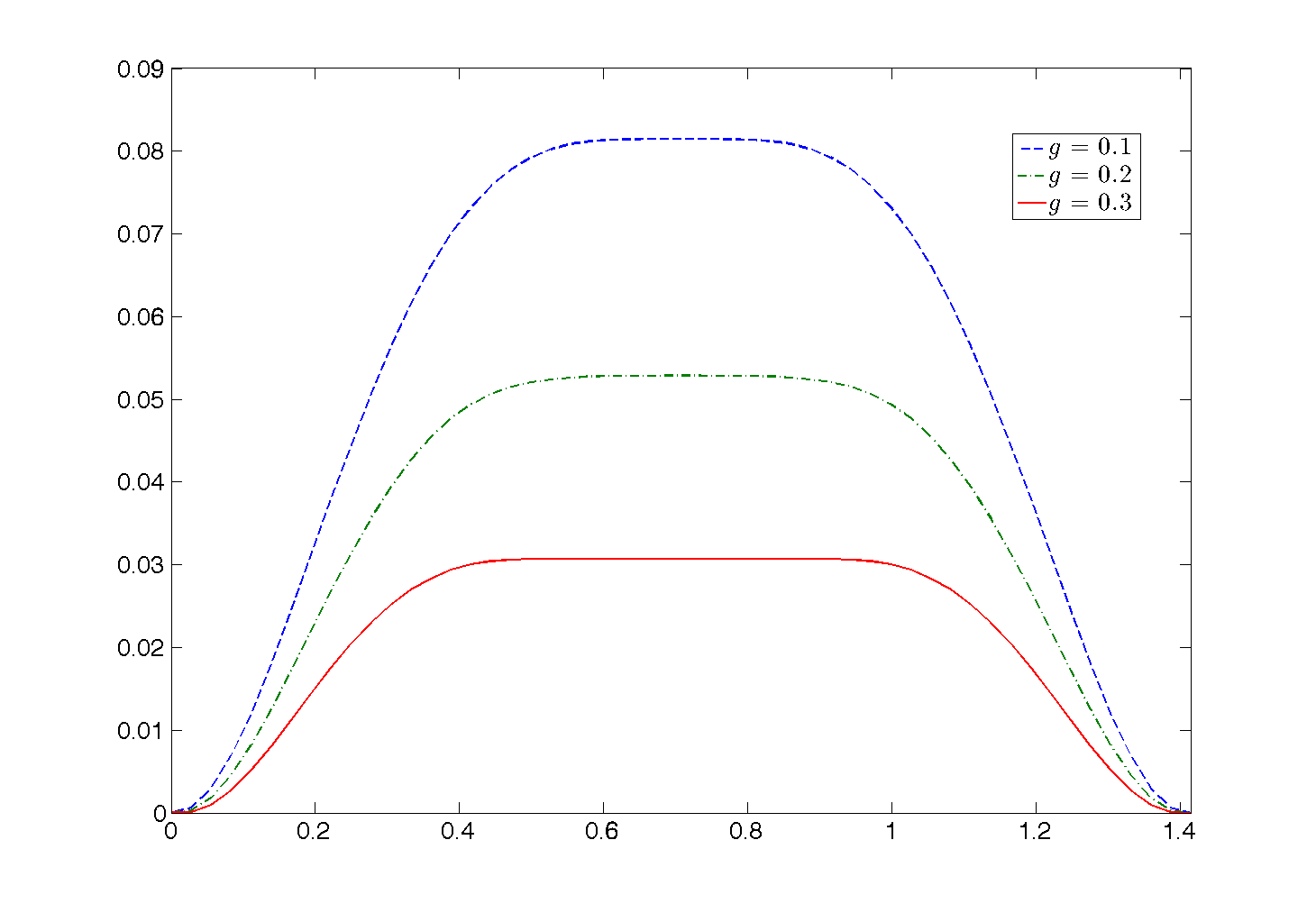}
\end{center}
\caption{Calculated $u$ for $p=1.5$ and $g=0.1$ (top, left), $g=0.2$ (top, right) and $g=0.3$ (down, left). Velocity profile for the calculated velocities along the diagonal of the square (down, right).  Parameters: $\gamma=10^3$.}\label{fig:up15}
\end{figure}

The resulting velocity functions and the velocity profiles along the diameter of the pipe are displayed in Figure \ref{fig:up15}. As in the previous case, the shear stress transmitted by a fluid layer decreases toward the center of the pipe which provokes the solid-like movement in that sector. Further, it is expected that if the value of $g$ increases, the flow tends to slow down and the flat zones tend to be bigger. This is clearly shown in the figures depicted, which are in good agreement with previous contributions (\textit{e.g.},\cite{Huilgol}). 


\subsection{Numerical Results: Case $p> 2$}
In this section, we focus on the behavior of Algorithm \ref{algodisp>2}. In the next experiments, we consider that the problem \eqref{eq:probreg} represents the flow of a Herschel-Bulkley fluid with $p> 2$, so we are in the case of a shear-thickening material (see \cite{Huilgol}). Further, we consider a constant $f$, which represents the linear decay of pressure in the pipe. As in the previous section, the constant $g$ plays the role of the Oldroy number. For further details in the mechanics of these problems, we refer the reader to \cite{Chhabra,dlRGpf,Huilgol} and the references therein. 

We initialize the algorithm \ref{algodisp>2} with the solution of the Poisson problem
$-\Delta u_0^h=f^h$, and we terminate the iterations according to the stopping criteria described in Remark \ref{rem:stop}.

As in the previous section, we use uniform triangulations described by $h$, the radius of the inscribed circumferences of the triangles in the mesh. 

It is remarkable to state that the classical $p$-Laplacian problem (\textit{i.e.}, \eqref{eq:probreg} with $g=0$) is difficult to solve when $p+1/(p-1)$ is large. This issue needs to be take into account in our case too (see \cite{Huang}). 

\subsubsection{Experiment 1}
In this experiment, we set $\Omega$ to be the unit square, and we compute the flow of a Herschel-Bulkley material with $p=4$. We analyze the behavior of the algorithm with $g=0.2$ and $f=3$. We work with a mesh given by $h\approx 0.0029$, and we use the value $\gamma=10^{3}$.

\begin{figure}
\begin{center}
\includegraphics[width=80mm, height=60mm]{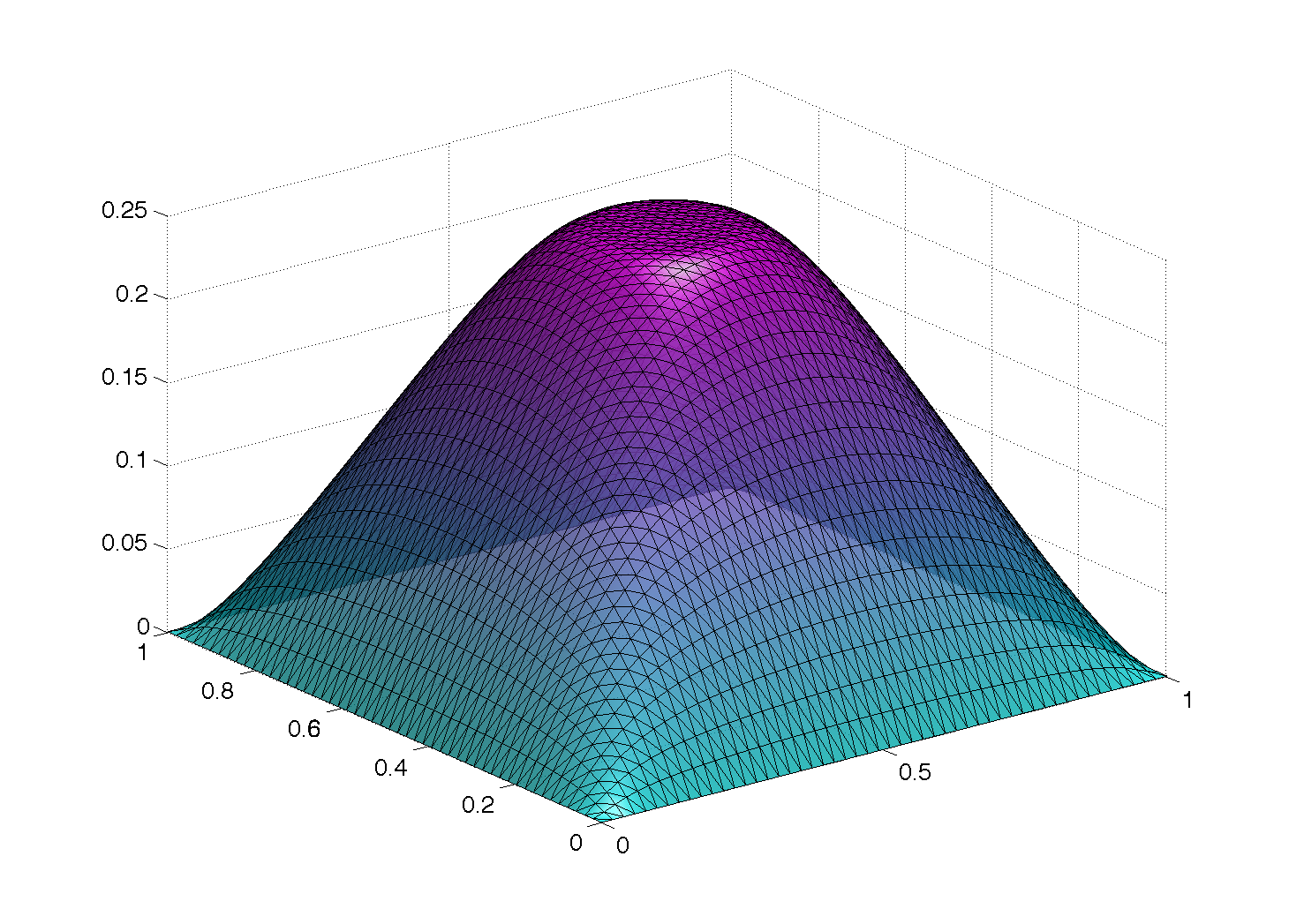}\hspace{0.cm}\includegraphics[width=75mm, height=55mm]{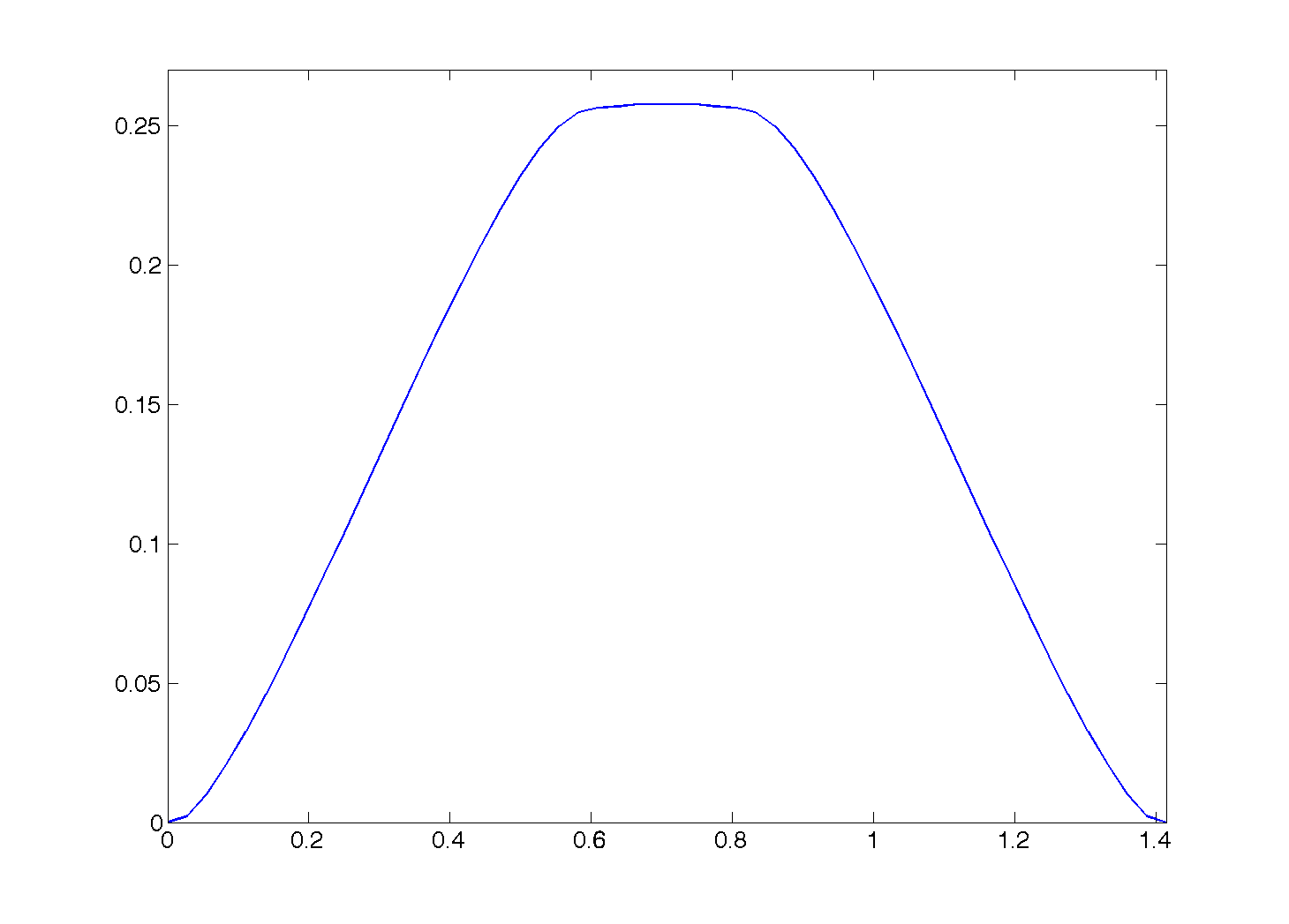}
\end{center}
\caption{Calculated $u$ for $p=4$ (left) and velocity profile along the diagonal of the pipe (right). Parameters: $g=0.2$ and $\gamma=10^3$.}\label{fig:up4}
\end{figure}

The resulting velocity function and the velocity profile along the diagonal of the square pipe are displayed in Figure \ref{fig:up4}. The graphics illustrate the expected mechanical properties of the material: the viscosity of shear-thickening materials increases with the rate of shear strain. In this case, since the shear stress transmitted by a fluid layer decreases toward the center of the pipe, the velocity takes a conical form with a flat part in the exact center of the geometry. 

In Table \ref{tab:p4} we show the number of iterations that Algorithm \ref{algodisp>2} needs to achieve convergence. We also show the evolution of $|J'_{\gamma,h}(\overrightarrow{u}_k)|/|J'_{\gamma,h}(\overrightarrow{u}_0)|$, $J_{\gamma,h}(\overrightarrow{u}_k)$, $\alpha_k$ and the number of inner iterations needed by Algorithm \ref{algo:armijo} to achieve convergence. The Algorithm performs as expected, \textit{i.e.}, the value of the functional is monotonically reduced at every iteration. Further, the residual behaves typically as in a deepest descent algorithm, but it shows fast local convergence. This behaviour can be appreciated in Figure \ref{fig:errp4}.  This fact can be explained due to the stronger regularity of the differential operator when $p>2$. As soon as $g$ increases, this effect will be lost. This will be shown in the next experiment. 

\begin{figure}
\begin{center}
\includegraphics[width=70mm, height=50mm]{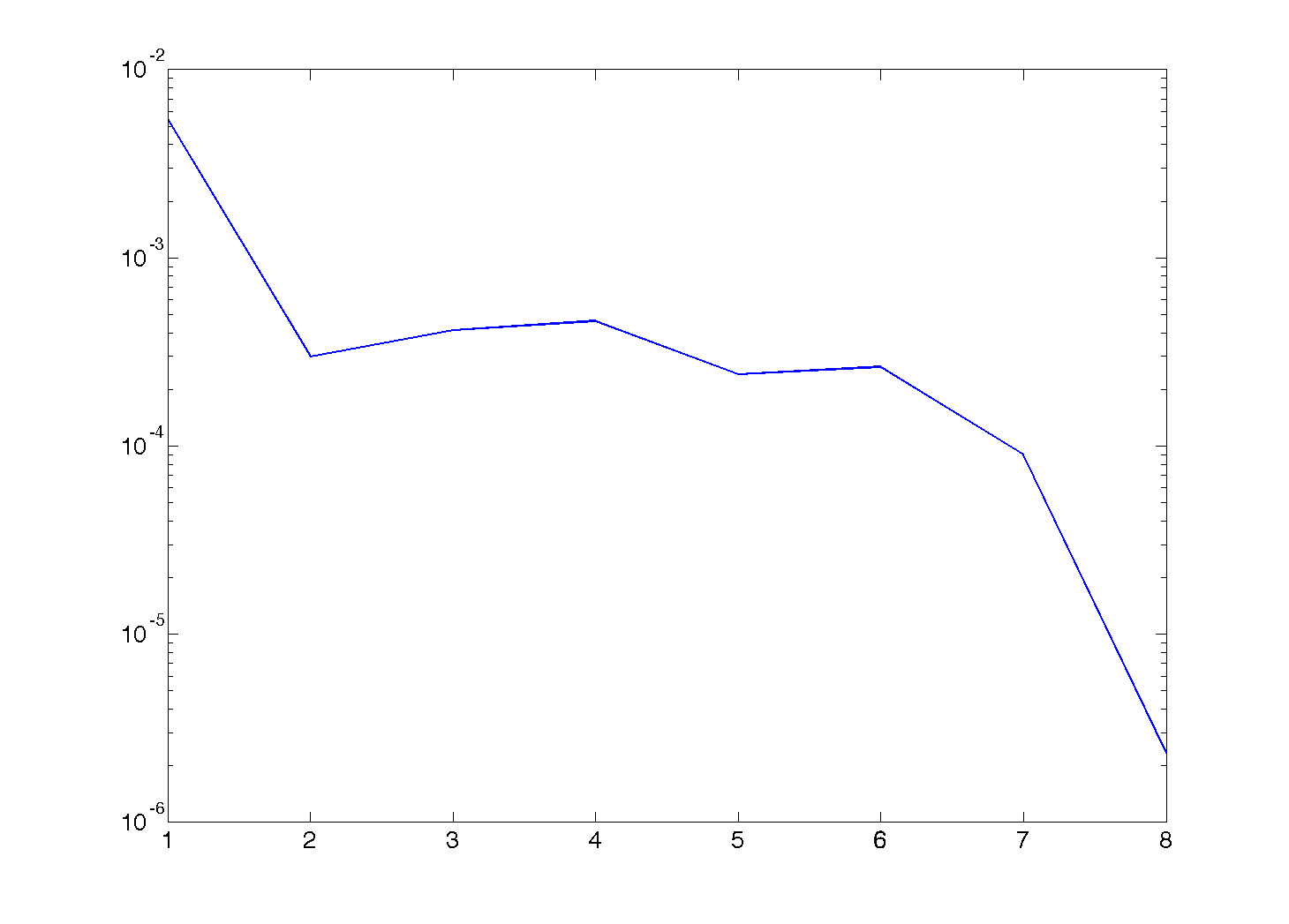}
\end{center}
\caption{Calculated residuals $|J'_{\gamma,h}(u_k)|/|J_{\gamma,h}'(u_0)|$, for $p=4$ and $g=0.1$. Parameters: $\gamma=10^3$.}\label{fig:errp4}
\end{figure}

Next, note that the step $\alpha_k$ does not have a monotone evolution during all the iterations. Also, the line search Algorithm \ref{algo:armijo}, for some iterations, needs no inner iterations to achieve convergence. These facts can be explained due to the stronger convexity that the functional exhibits when $p>2$, which implies that $\varphi_k(\alpha)=J(u_k+\alpha w_k)$ is better approximated by the quadratic model $m_k$.

\begin{table}
\begin{center}
\begin{tabular}{|c|c|c|c|c|}
\hline
it. & $|J'_{\gamma,h}(\overrightarrow{u}_k)|/|J'_{\gamma,h}(\overrightarrow{u}_0)|$ & $J_{\gamma,h}(\overrightarrow{u}_k)$ & $\alpha_k$ & l.s. it.\\
\hline
 1 & 5.4526e-3 & -0.17950 & 1.0000 & 0\\ 
 2 & 2.9779e-4 & -0.18069 & 0.4974 & 1\\
 3 & 4.1359e-4 & -0.18089 & 1.0000 & 0\\
 4 & 4.6274e-4 & -0.18101 & 0.3882 & 1\\
 5 & 2.4101e-4 & -0.18107 & 0.2402 & 1\\
 6 & 2.6292e-4 & -0.18108 & 0.0750 & 2\\
 7 & 8.9793e-5 & -0.18109 & 0.0333 & 3\\
 8 & 2.3358e-6 & -0.18109 & 0.0375 & 2\\
  \hline
\end{tabular}
\caption{Convergence behavior for Algorithm \ref{algodisp>2}. Parameters: $p=4$, $g=0.2$ and $\gamma=10^3$}\label{tab:p4}
\end{center}
\end{table}


\subsubsection{Experiment 2}
In this experiment, we set $\Omega$ to be the unit ball, and we compute the flow of a Herschel-Bulkley material with $p=10$. We analyze the behavior of the algorithm with $f=1$ and compare the performance of the Algorithm for $g=0.1$ and $g=0.4$.

\begin{figure}
\begin{center}
\includegraphics[width=70mm, height=50mm]{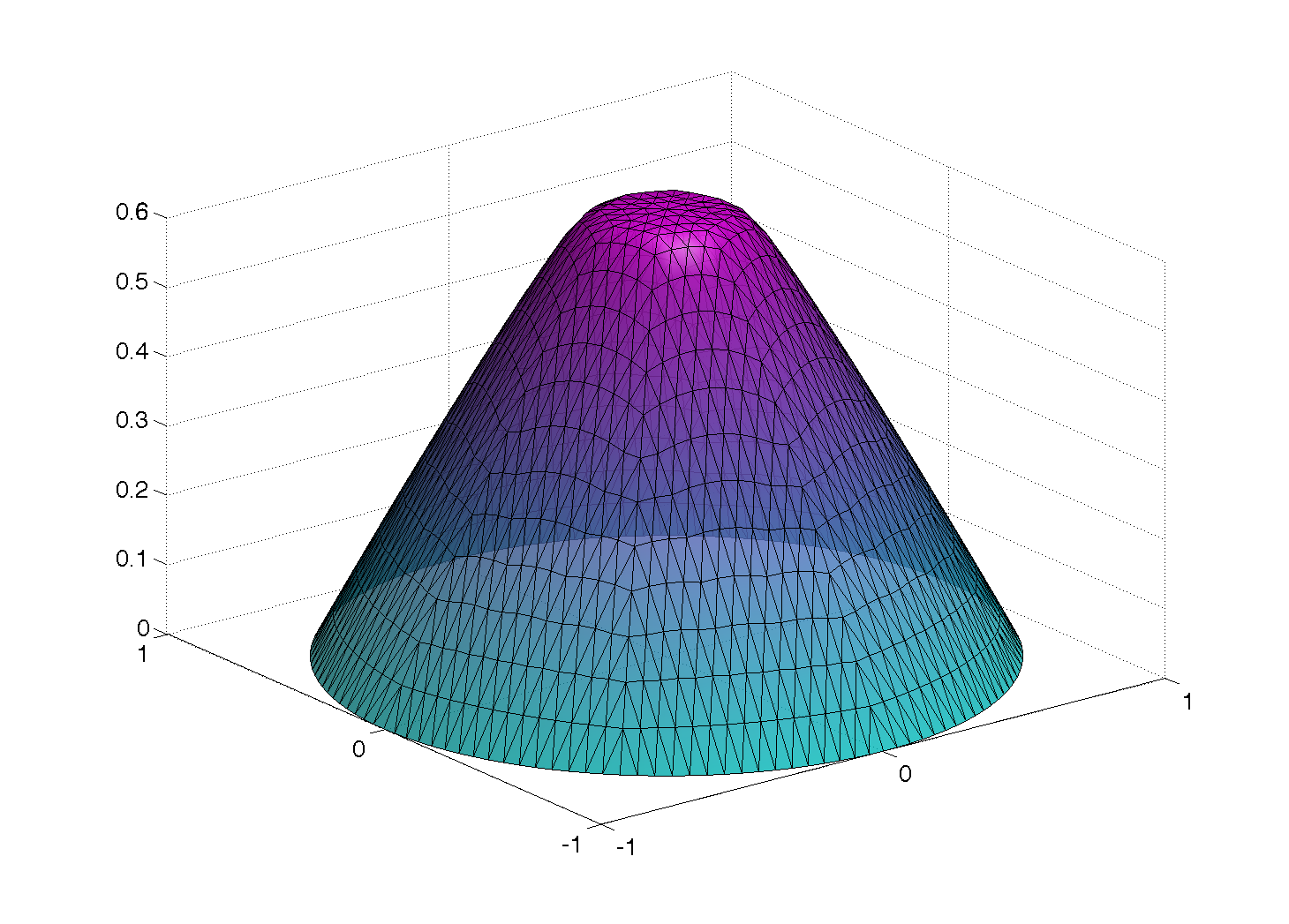}\hspace{0.cm}\includegraphics[width=70mm, height=50mm]{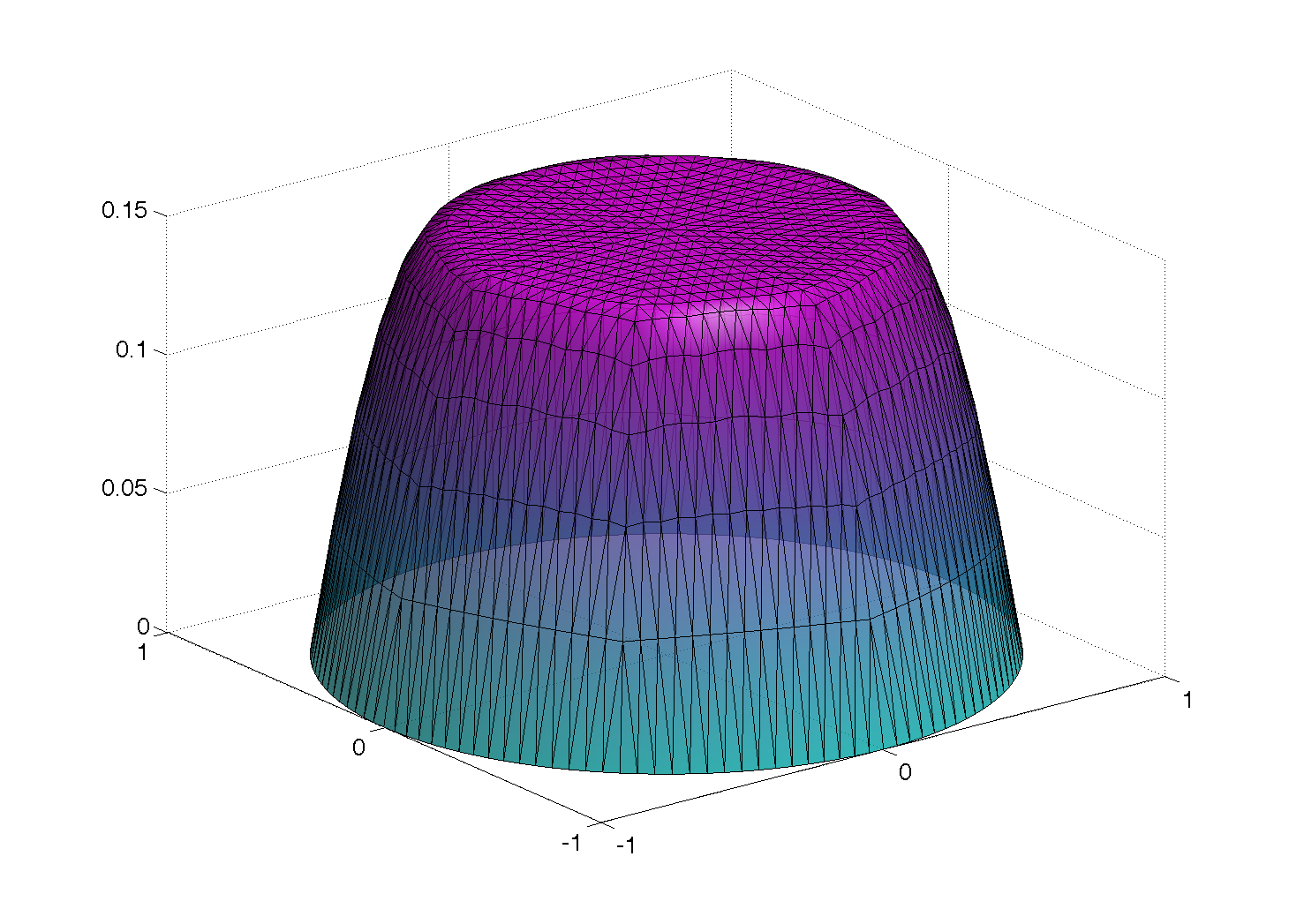}
\end{center}
\caption{Calculated $u$ for $p=10$ and $g=0.1$ (left) and $g=0.4$ (right). Parameters: $\gamma=10^3$.}\label{fig:up10}
\end{figure}

\begin{figure}
\begin{center}
\includegraphics[width=70mm, height=50mm]{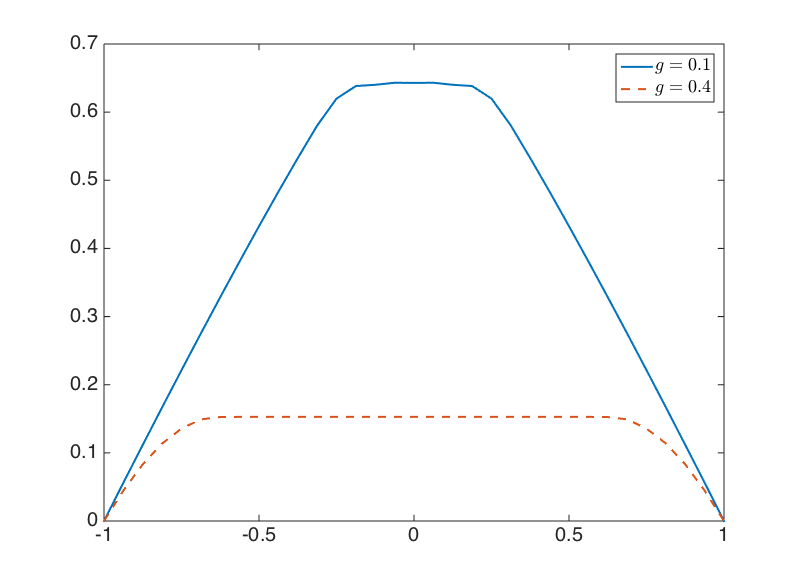}
\end{center}
\caption{Velocity profiles for $p=10$.}\label{fig:profp10}
\end{figure}

In Figure \ref{fig:up10} the calculated velocities for $g=0.1$ and $g=0.4$ are depicted. As stated in the previous experiment, small values of $g$ make the problem be close to the classical $p$-Laplacian problem. In this case, the Algorithm exhibits good performance. On the other hand, bigger values of $g$ make the problem less regular. Also, from the mechanical point of view if the values of $g$ increase, the size of the inactive zones increases as well. Therefore, the problem is more difficult to be approximated (see \cite{dlRGpf,Huilgol}). 

\begin{figure}
\begin{center}
\includegraphics[width=70mm, height=50mm]{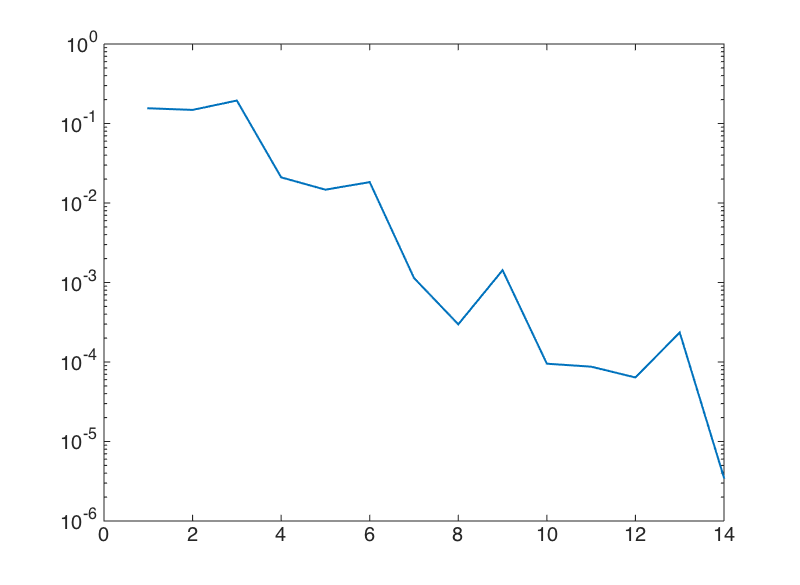}\hspace{0.5cm}\includegraphics[width=70mm, height=50mm]{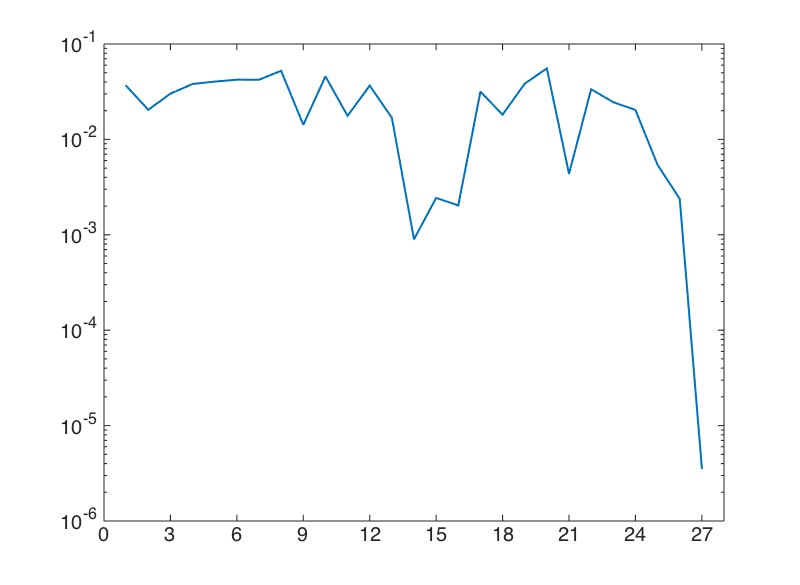}
\end{center}
\caption{Calculated residual $|J'_{\gamma,h}(u_k)|/|J_{\gamma,h}'(u_0)|$ for $p=10$ and $g=0.1$ (left) and $g=0.4$ (right). Parameters: $\gamma=10^3$.}\label{fig:errp10}
\end{figure}

Regarding the performance of the Algorithm, in Figure \ref{fig:errp10} the evolution of the error for $g=0.1$ and $g=0.4$ are depicted. For $g=0.1$, the error evolves in a typical way and the Algorithm achieves convergence in 14 iterations. On the other hand, for $g=0.4$ the error is very oscillating and the Algorithm needs 27 iterations to achieve convergence. As expected, the fact that $p$ and $g$ increase provokes instabilities in the algorithm. 

Finally, in Table \ref{tab:exp>2mesh} we show the behaviour of the algorithm \ref{algodisp>2} in different meshes, considering a fixed value of $g=0.1$. Here it can be appreciated that the Algorithm requires more iterations to achieve convergence as the mesh gets finer. This fact suggests that the Algorithm is not mesh independent. This is not shocking news, since the convergence result for Algorithm \ref{algodisp>2} was obtained in a finite dimensional space. However, the algorithm still requires relatively few iterations to produce reliable solutions with low computational cost.

\begin{table}
\begin{center}
\begin{tabular}{|c|c|c|c|c|}
\hline
$g=0.1$ & $h_1$ & $h_2$ & $h_3$ & $h_4$\\
\hline
Iter. num. & 13 & 14 & 22 & 37 \\
$|J'_{\gamma,h}(\overrightarrow{u}_k)|/|J'_{\gamma,h}(\overrightarrow{u}_0)|$ & 5.242e-7 & 3.868e-4  & 3.931e-6  & 4.518e-6\\
$J_{\gamma,h}(\overrightarrow{u}_k)$ & -0.550987 & -0.582629 &  -0.590403 & -0.592738 \\
\hline
\end{tabular}
\caption{Convergence behavior for Algorithm \ref{algo1<p<2}. Parameters: $p=10$, $g=0.1$ and $\gamma=10^3$}\label{tab:exp>2mesh}
\end{center}
\end{table}

\subsubsection{Experiment 3}
One key issue in our approach is the size of the regularization parameter. In fact, theoretically, we obtain a better approximation for the problem when $\gamma$ is big. However, it is not a good strategy to directly run the Algorithms with high values for the parameter, since instabilities can arise in the process. In order to help the regularization parameter reach high values, we perform a simple but effective continuation technique: given $\gamma_k$, we run the algorithm and obtain the corresponding solution $\overrightarrow{u}^h_{\gamma_k}$. Next, we set $\gamma_{k+1}=10\gamma_k$, initialize the algorithm with $\overrightarrow{u}^h_{\gamma_k}$ and run it to obtain $\overrightarrow{u}^h_{\gamma_{k+1}}$. We stop this process when $\gamma$ equals $10^6$. 

\begin{figure}
\begin{center}
\includegraphics[width=70mm, height=50mm]{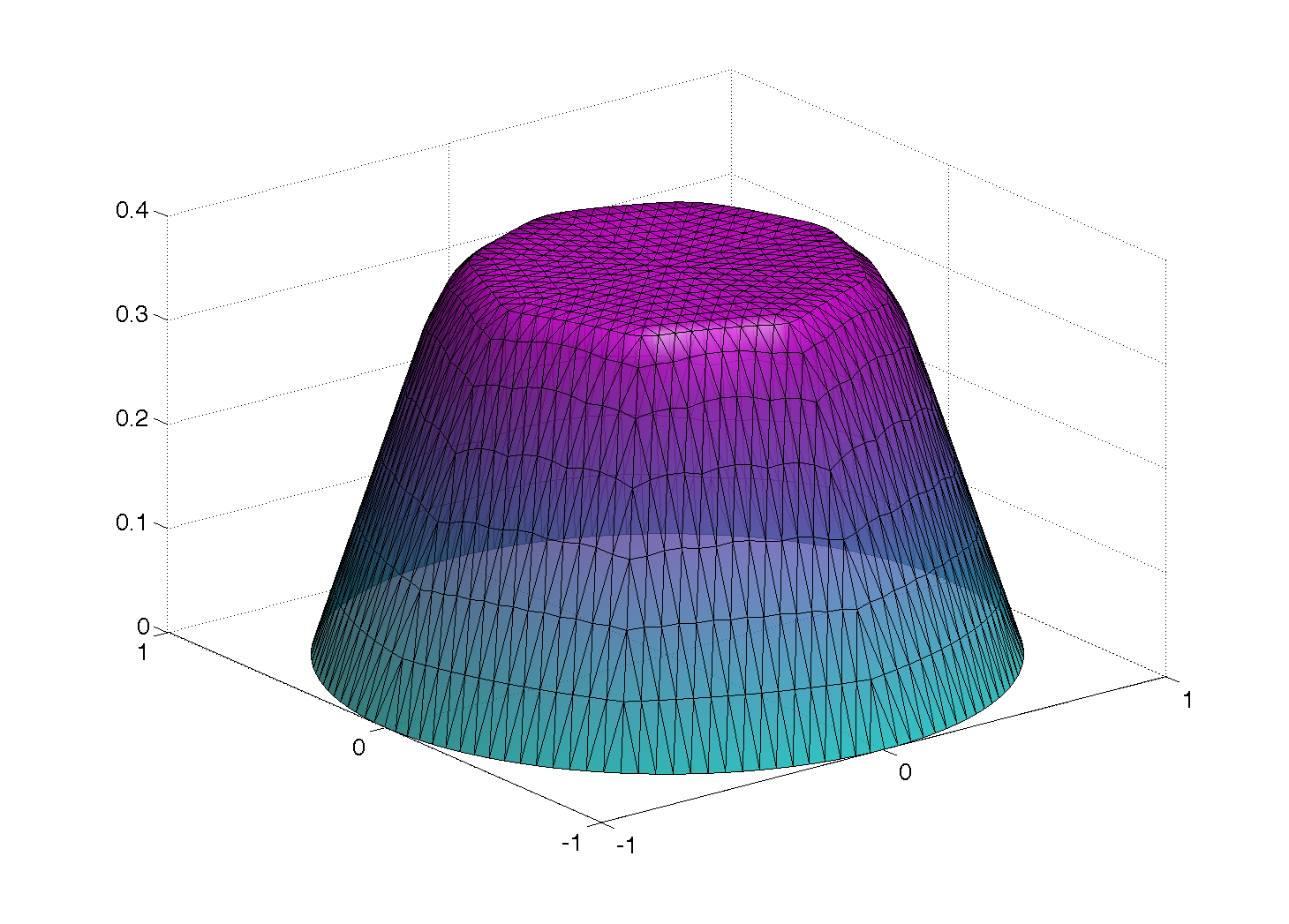}
\end{center}
\caption{Calculated velocity $u$ for $p=100$ and $g=0.3$ at $\gamma=10^5$.}\label{fig:up100}
\end{figure}

As stated before, a challenging problem when using the $p$-Laplacian operator arises when $p$ is big. Therefore, we are interested in the computation of the flow of a Herschel-Bulkley material with $p=100$ and $g=0.3$. If we set $\gamma=10^6$ and run the Algorithm \ref{algodisp>2}, convergence is not achieved. However, by using the continuation strategy, we obtain the solution for this problem. 

The convergence history is shown in Table \ref{tab:p100}. We show the number of iterations that the Algorithm \ref{algodisp>2} needs to achieve convergence for each $\gamma_k$, the value of the functional $J_{\gamma,h}(\overrightarrow{u}_{\gamma_k}^h)$ and the norm of the calculated velocity $\overrightarrow{u}_{\gamma_k}^h$. It is possible to observe that, although the continuation technique helps the algorithm achieve convergence for high values of $\gamma_k$, the value of the functional and the norm of the velocity stabilize as soon as $\gamma_k$ equals $10^3$. This fact can be explained since the regularization procedure is sharp. Thus, for values around $\gamma_k=10^3$ provides reliable results for the problem. However, we think that a further research in path following methods can clarify these aspects (see \cite{dlRGpf}). 

\begin{table}
\begin{center}
\begin{tabular}{|c|c|c|c|c|c|c|}
\hline
$\gamma_k$ & $10^1$ & $10^2$ & $10^3$ & $10^4$ & $10^5$ & $10^6$\\
\hline
Iter. num. & 35 & 11 & 5 & 9 & 2 & 1 \\
$J_{\gamma,h}(\overrightarrow{u}^h_{\gamma_k})$ & -0.2452 & -0.2354 &  -0.2344 & -0.2343 & -0.2343 & -0.2343 \\
$\|\overrightarrow{u}^h_{\gamma_k}\|_{p,h}$ & 9.0090 & 9.0066 & 9.0063 & 9.0063 & 9.0063& 9.0063\\
\hline
\end{tabular}
\caption{Behavior of the continuation technique. Parameters: $p=100$ and $g=0.4$}\label{tab:p100}
\end{center}
\end{table}

\section{Conclusions}\label{sec:conclusions}
In this paper, we focused on the numerical resolution of a class of variational inequalities of the second kind involving the $p$-Laplacian operator and the $L^1$-norm of the gradient. The non differentiability of the associated functional was overcame with a Huber regularization procedure. This kind of local regularization has proved to be efficient in the context of this kind of problems. Based on optimization and variational techniques, we proposed preconditioned descent algorithms for coping the two cases $1<p<2$ and $p>2$. For the first case, we proposed an infinite dimensional descent algorithm and proved a global convergence result for it. The second case posed a difficult analytical issue, due to the lack of regularity of the candidates for descent directions. Thus, we proposed an algorithm in a finite dimensional setting and proved a global convergence result for this algorithm as well. Several numerical experiments were carried out to show the main features of the numerical approach. These numerical examples were constructed focusing on the applications to the flow of Herschel-Bulkley materials. Due to the structure of all the algorithms proposed, it was only necessary to solve one linear system at each iteration of the algorithms. This fact implied a low computational cost for all our numerical realisation. 

In order to continue this research, we consider that a deeper analysis of the case $p>2$ is an interesting perspective. Here, the use of $H^s$ spaces provides a promising way to follow. Also, the combination of this approach with multigrid algorithms will be useful in order to cope more challenging problems, such as the $p$-Stokes problem (2D and 3D flows of Herschel-Bulkley materials). Finally, the analysis and simulation of blood flow models involving the Herschel-Bulkley structure looks like a very promising field of research.
\begin{acknowledgements}
I would like to thank Prof. Dr. Juan Carlos De los Reyes (ModeMat-Quito) and Prof. Dr. Eduardo Casas (Univ. de Cantabria-Spain) for all the helpful discussions and good insights in the problem. I also would like to thank the anonymous referees for many helpful comments which lead to a significant improvement of the article. Finally, thanks to Prof. Dr. Michael Hinze, Prof. Dr. Winniefred Wollner and Prof. Dr. Ingenuin Gasser for the kind hospitality and interesting discussions during my stay in Hamburg Universit\"at.\end{acknowledgements}


\begin{thebibliography}{99}
\bibitem{jidiaz}
{\sc S. N. Antontsev, J. I. D\'iaz and S. Shmarev}, {\em Energy Methods for Free Boundary Problems. Applications to Nonlinear PDEs and Fluid Mechanics}, Birkh\"auser, USA, 2002.

\bibitem{barret}
{\sc J. W. Barrett and W. B. Liu}, {\em Finite Element Approximation of the $p$-Laplacian}, Mathematics of Computation, 61 (1993) 523--537.

\bibitem{bermejomg}
{\sc R. Bermejo and J. A. Infante}, {\em A Multigrid Algorithm for the $p$-Laplacian}, SIAM J. Sci. Comput. , 21 (2000) 1774--1789 .

\bibitem{brezis}
{\sc H. Brezis}, {\em Functional Analysis, Sobolev Spaces and Partial Differential Equations}, Springer, USA, 2011.


\bibitem{cc50}
{\sc J. Alberty, C. Carstensen and S.A. Funken}, {\em Remarks Around 50 Lines of Matlab: Short Finite Element Implementation}, Numerical Algorithms, 20 (1999) 117--137.

\bibitem{casas}
{\sc E. Casas and L. A. Fern\'andez}, {\em Distributed Control of Systems Governed by a General Class of Quasilinear Elliptic Equations}, Journal of Differential Equations, 104 (1993) 20--47.

\bibitem{chenqi}
{\sc X. Chen, Z. Nashed and L. Qi}, {\em Smoothing Methods and Semismooth Methods for Nondifferentiable Operator Equations}, SIAM J. Numer. Anal., 38 (2000) 1200--1216.

\bibitem{Chhabra}
{\sc R. P. Chhabra and J. F. Richardson}, {\em Non-Newtonian Flow and Applied Rheology}, Elsevier, Hungary, 2008.

\bibitem{coffman}
{\sc C. V. Coffman, V. Duffin and V. J. Mizel}, {\em Positivity of Weak Solutions of Non Uniformly Elliptic Equations}, Ann. Mat. Pura Appl., 104 (1975) 209--238.

\bibitem{daulions}
{\sc R. Dautray and J. L. Lions}, {\em Mathematical Analysis and Numerical Methods for Science and Technology. Volume 2: Functional Analysis and Variational Methods}. Springer-Verlag, Germany, 2000.

\bibitem{DelosR}
{\sc J. C. De los Reyes}, {\em Numerical PDE-Constrained Optimization}, Springer, 2015.

\bibitem{dlRGpf}
{\sc J. C. De los Reyes and S. Gonz\'alez}, {\em  Path Following Methods for Steady Laminar Bingham Flow in Cylindrical Pipes}, Mathematical Modelling and Numerical Analysis, 43 (2009) 81-117.

\bibitem{dlRG2D}
{\sc J. C. De los Reyes and S. Gonz\'alez Andrade}, {\em Numerical simulation of two-dimensional Bingham fluid flow by semismooth Newton methods}, Journal of Computational and Applied Mathematics, 235 (2010) 11--32.

\bibitem{dlRGtd}
{\sc J. C. De los Reyes and S. Gonz\'alez Andrade}, {\em A combined BDF-semismooth Newton approach for time-dependent Bingham flow}, Numerical Methods for Partial Differential Equations, 28 (2012) 834--860.

\bibitem{dlRHint}
{\sc J. C. De los Reyes and M. Hinterm\"uller}, {\em A Duality Based Semismooth Newton Framework for Solving Variational Inequalities of the Second Kind}, Interfaces and Free Boundaries, 13 (2011), 437–-462.

\bibitem{densch}
{\sc J. E. Dennis and R. B. Schnabel}, {\em Numerical Methods for Unconstrained Optimization and Nonlinear Equations}. SIAM, U.S.A, 1996.

\bibitem{ektem}
{\sc I. Ekeland and R. Temam}, {\em Convex Analysis and Variational Problems}. North-Holland Publishing Company, The Netherlands, 1976.

\bibitem{kanzow}
{\sc C. Geiger and C. Kanzow}, {\em Numerische Verfahren zur L\"osung unrestringierter Optimierungsaufgaben}. Springer, Deutschland, 1999.

\bibitem{glomaro}
{\sc R. Glowinski and A. Marroco}, {\em  Sur L'Approximation par Elements Finis d'Ordre Un, et la Resolution, par Penalisation-Dualite, d'une Classe de Problemes de Dirichlet non Lineaires}, R.A.I.R.O, 9 (1975) 41-76.

\bibitem{groeger}
{\sc K. Gr\"oger}, {\em  A $W^{1,p}$-Estimate for Solutions to Mixed Boundary Value Problems for Second Order Elliptic Differential Equations}, Mathematische Annalen, 283 (1989) 679--687.

\bibitem{hik1}
{\sc M.~Hinterm\"uller, and K.~Ito and K.~Kunisch},
{\em The primal-dual active set strategy as a semi-smooth Newton method},
SIAM J. Opt., 13 (2003), pp.~865--888.

\bibitem{hinrau1}
{\sc M. Hinterm\"uller and C. Rautenberg}, {\em A Sequential Minimization Technique for Elliptic Quasi-Variational Inequalities with Gradient Constraints}, SIAM J. Optim, 22 (2012) 1224--1257.

\bibitem{hinrau2}
{\sc M. Hinterm\"uller and C. Rautenberg}, {\em Parabolic Quasi-Variational Inequalities with Gradient-Type Constraints}, SIAM J. Optim, 23 (2013) 2090--2123.

\bibitem{hinetal}
{\sc M. Hinze, R. Pinnau, M. Ulbrich and S. Ulbrich}, {\em Optimization with PDE Constraints}. Springer, 2009.

\bibitem{Huang}
{\sc Y. Q. Huang, R. Li and W. Liu}, {\em Preconditioned Descent Algorithms for p-Laplacian}, Journal of Scientific Computing, 32 (2007) 343--371.

\bibitem{Huilgol}
{\sc R. R. Huilgol and Z. You}, {\em Application of the Augmented Lagrangian Method to Steady Pipe Flows of Bingham, Casson and Herschel-Bulkley Fluids}, J. Non-Newtonian Fluid Mech. 128 (2005) 126--143.

\bibitem{jahn}
{\sc J. Jahn}, {\em Introduction to the Theory of Nonlinear Optimization}. Springer-Verlag, Germany, 2007.

\bibitem{joubue}
{\sc G. Jouvet and E. Bueler}, {\em Steady, Shallow Ice Sheets as Obstacle Problems: Well-Posedness and Finite Element Approximation}, SIAM J. Optim. 23 (2013) 2090--2123.

\bibitem{kelley}
{\sc C. T. Kelley}, {\em Iterative Methods for Optimization}. SIAM, U.S.A., 1999.

\bibitem{liebloss}
{\sc E. H. Lieb and M. Loss}, {\em Analysis}. AMS, U.S.A., 2001.

\bibitem{lions}
{\sc J. L. Lions}, {\em Optimal Control of Systems Governed by Partial Differential Equations}. Springer-Verlag, Germany, 1971.

\bibitem{liubarret}
{\sc W. B. Liu and J. W. Barret}, {\em Quasi-Norm Error Bounds for the Finite Element Approximation of Some Degenerate Quasilinear Elliptic Equations and Variational Inequalities}, ESAIM: Mathematical Modelling and Numerical Analysis, 28 (1994) 725--744.

\bibitem{nocacta}
{\sc J. Nocedal}, {\em Theory of Algorithms for Unconstrained Optimization}, Acta Numerica, 1 (1992) 199--242.

\bibitem{Quart-blood}
{\sc A. Quarteroni, M. Tuveri  and A. Veneziani}, {\em Computational vascular fluid dynamics: problems, models and methods}, Computing and Visualization in Science, 2 (2000) 163--197.

\bibitem{Sankar-blood}
{\sc D. S. Sankar and Usik Lee}, {\em Two-fluid Herschel-Bulkley Model for Blood Flow in Catheterized Arteries}, Journal of Mechanical Science and Technology, 22 (2008) 1008--1018.

\bibitem{Shah-blood}
{\sc S. R. Shah}, {\em An Innovative Study for non-Newtonian Behaviour of Blood Flow in Stenosed Artery using Herschel-Bulkley Fluid Model}, International Journal of Bio-Science and Bio-Technology, 5 (2013) 233--240.

\bibitem{Simader}
{\sc C. G. Simader}. {\em On Dirichlet's Boundary Value Problems}. Lecture Notes in Mathematics, No. 268. Springer, Germany, 1972.

\bibitem{Simon1}
{\sc J. Simon}. 1978. Regularit\'e de la Solution d'une Equation non Lineaire dans $\re{R}^N$. In: P. Benilan ed. Lecture Notes in Mathematics, No. 665. Springer, pp. 205-227.

\bibitem{Struwe}
{\sc M. Struwe}. {\em Variational Methods. Applications to Nonlinear Partial Differential Equations and Hamiltonian Systems}. Springer. Germany. 2008.

\bibitem{SunYuan}
{\sc W. Sun and Y. -X. Yuan}, {\em Optimization Theory and Methods. Nonlinear Programming}. Springer, U.S.A., 2006.

\bibitem{Triebel}
{\sc H. Triebel}, {\em Interpolation Theory, Function Spaces, Differential Operators.} North Holland Publishing Company. GDR. 1978.

\bibitem{Trudinger}
{\sc N. S. Trudinger}, {\em Linear Elliptic Operators with Measurable Coefficients}, Ann. Scuola Norm. Sup. Pisa, 27 (1973) 265--308.

\bibitem{ulbrich}
{\sc M. Ulbrich}, {\em Semismooth Newton Methods for Operator Equations in Function Spaces}, SIAM J. Optim., 13 (2003) 805--841.
\end{thebibliography}
\end{document}